\documentclass[a4paper,12pt]{amsart}
\usepackage[T1]{fontenc}
\usepackage[centering,margin=2.3cm]{geometry}
\usepackage[latin1]{inputenc}
\usepackage{amsthm, amsmath}   
\usepackage{latexsym,amssymb}
\usepackage{algorithmic}
\usepackage{amssymb}

\usepackage{amsaddr}
\usepackage{graphicx,subfigure}
\usepackage{times}
\usepackage{enumerate,url}
\usepackage[usenames,dvipsnames]{pstricks}
\usepackage{epsfig}
\usepackage{pst-node}
\usepackage{pst-grad} 
\usepackage{pst-plot} 
\usepackage{pst-3dplot}
\usepackage{color}
\usepackage{cite}
\newcommand{\Cay}{\mathrm{Cay}}

\newcommand{\even}{\mathrm{even}}
\newcommand{\odd}{\mathrm{odd}}
\newcommand{\uparr}{\ensuremath{\hspace{2pt}\uparrow\hspace{-3pt}}}
\newcommand{\dwnarr}{\ensuremath{\hspace{2pt}\downarrow\hspace{-3pt}}}

\theoremstyle{plain}
\newtheorem{theorem}{Theorem}[section]
\newtheorem{lemma}[theorem]{Lemma}
\newtheorem{corollary}[theorem]{Corollary}
\newtheorem{proposition}[theorem]{Proposition}

\theoremstyle{definition}

\newtheorem{conjecture}[theorem]{Conjecture}

\newtheorem{question}[theorem]{Question}
\newtheorem{remark}[theorem]{Remark}

\title[Sensitivity]{On sensitivity in bipartite Cayley graphs}
\author[I. Garc\'{i}a-Marco]{Ignacio Garc\'{i}a-Marco}
\address{Facultad de Ciencias, Universidad de La Laguna, La Laguna, Spain}
\author[K. Knauer]{Kolja Knauer$^*$}
\address{Aix Marseille Univ, Universit\'e de Toulon, CNRS, LIS, Marseille, France\\Departament de Matem\`atiques i Inform\`atica,
Universitat de Barcelona, Spain}

\keywords{  \\ \ \ \ $ ^*$ Corresponding author }


\begin{document}

\begin{abstract}
 Huang proved that every set of more than half the vertices of the $d$-dimensional hypercube $Q_d$ induces a subgraph of maximum degree at least $\sqrt{d}$, which is tight by a result of Chung, F\"uredi, Graham, and Seymour. Huang asked whether similar results can be obtained for other highly symmetric graphs.
 

First, we present three infinite families of Cayley graphs of unbounded degree that contain induced subgraphs of maximum degree $1$ on more than half the vertices. In particular, this refutes a conjecture of Potechin and Tsang, for which first counterexamples were shown recently by Lehner and Verret. The first family consists of dihedrants and contains a sporadic counterexample encountered earlier by Lehner and Verret. The second family are star graphs, these are edge-transitive Cayley graphs of the symmetric group. All members of the third family are $d$-regular containing an induced matching on a $\frac{d}{2d-1}$-fraction of the vertices. This is largest possible and answers a question of Lehner and Verret.

Second, we consider Huang's lower bound for graphs with subcubes and show that the corresponding lower bound is tight for products of Coxeter groups of type $\mathbf{A_n}$, $\mathbf{I_2}(2k+1)$, and most exceptional cases. We believe that Coxeter groups are a suitable generalization of the hypercube with respect to Huang's question.

Finally, we show that induced subgraphs on more than half the vertices of Levi graphs of projective planes and of the Ramanujan graphs of Lubotzky, Phillips, and Sarnak have unbounded degree. This gives classes of Cayley graphs with properties similar to the ones provided by Huang's results. However, in contrast to Coxeter groups these graphs have no subcubes.
\end{abstract}

\maketitle

\section{Introduction}\label{introduction}

Recently, Huang~\cite{huang2019induced} proved the Sensitivity Conjecture~\cite{Sensitivity} by showing that an induced subgraph on more than half of the vertices of the $d$-dimensional hypercube $Q_d$ has maximum degree at least $\sqrt{d}$. For a graph $G=(V,E)$ denote by $\alpha(G)$ the size of a largest independent set in $G$, by $\Delta(G)$ its maximum degree, and for a $K\subseteq V$ by $G[K]$ the subgraph induced by $K$. Define the \emph{sensitivity} $\sigma(G)$ of $G$ as the minimum value $\Delta(G[K])$ among all the $K\subseteq V$ on more than $\alpha(G)$ vertices.
Since in a regular bipartite $G$ one has $\alpha(G)=\frac{|V|}{2}$, Huang's result can be expressed as $\sigma(Q_d)\geq\sqrt{d}$. Huang asks what can be said about $\sigma(G)$ if $G$ is a ``nice'' graph with high symmetry. Further, since by a result of Chung, F\"uredi, Graham, and Seymour~\cite{CFGS88} the bound for $Q_d$ is tight, he wonders for which graphs a tight bound on the sensitivity follows from his method.                                                                                                                                                                                                                                      

The present paper studies both of these questions by considering (simple, undirected\footnote{Even if graphs are considered undirected, in figures we use arcs to represent generators of order larger than $2$ to increase readability.}, right) \emph{Cayley graphs} of groups to be ``nice'' with high symmetry. That is, for a group $\Gamma$ and a subset $C\subseteq\Gamma$ define $\Cay(\Gamma,C)$ with $\{x,y\}\in E$ if and only if $x^{-1}y\in C$. First positive results in this direction were obtained by Alon and Zheng~\cite{alon2020unitary}, who proved that in a $d$-regular Cayley graph $G$ of an elementary abelian 2-group, then $\sigma(G) \geq \sqrt{d}$. Then recently, Potechin and Tsang~\cite{potechin2020conjecture} showed that for every $d$-regular Cayley graph $G$ of an abelian group any set of vertices of more than half the vertices induces a subgraph with maximum degree at least $\sqrt{{d}/{2}}$ -- hence answering Huang's question in the bipartite case.
Moreover, they conjectured this lower bound to hold for Cayley graphs of general groups. However shortly after, Lehner and Verret~\cite{verret2020counterexamples} found a bipartite cubic Cayley graph $G$ of a dihedral group with $\sigma(G)=1<\sqrt{3/2}$ -- thus, refuting the above conjecture. Moreover, they construct an infinite family of bipartite Cayley graphs of $2$-groups of unbounded degree, with $\sigma(G)=1$ for every member $G$ of the family. Thus, concerning Huang's questions, $\sigma(G)$ cannot be bounded from below by a function of the degree for general Cayley graphs.

\bigskip

In the first part of the present paper, we give three more \emph{insensitive} families of Cayley graphs, i.e., they have unbounded degree but $\sigma(G)=1$ for all their members $G$. 

The first family are bipartite \emph{dihedrants}, i.e., Cayley graphs of the dihedral group (Theorem~\ref{thm:mathchindihedral}). The smallest member of this family is the graph presented in~\cite[Section 3]{verret2020counterexamples} as well as the smallest non-cyclic, bipartite Cayley graph with $\sigma=1$ among all groups.

The second family are the \emph{star-graphs}~\cite{AK89}, i.e., Cayley graphs of $S_n$ with respect to all transpositions containing $1$. These graphs, that were initially motivated as an ``attractive alternative'' to the hypercube (see~\cite{ADK94}) form a family of bipartite edge-transitive Cayley graphs. The first non-trivial member is the \emph{Nauru graph} $G(12,5)$, see Figure~\ref{fig:homgraphstar} and~\cite{nauru} for a beautiful collection of models. Another feature that distinguishes this family from the previous one is that they are Cayley graphs with respect to a minimal set of generators of the group. We show that besides their very high symmetry star graphs have sensitivity $1$ (Theorem~\ref{thm:stargraph}). 

The third family consists of $d$-regular Cayley graphs that have an induced subgraph of maximum degree $1$  on a $\frac{d}{2d-1}$-fraction of the vertices (Theorem~\ref{thm:alternatingtight}). This is largest possible in a $d$-regular graph and settles a question posed in~\cite[Remark 2]{verret2020counterexamples}. In particular, we find the smallest such graphs and construct bipartite tight Cayley graphs by using the Kronecker double cover (Corollary~\ref{cor:bipdouble}). 

\bigskip

The second part of the paper concerns the question of when $\sigma$ can be bounded from below in a tight way. A first answer to this could be that many groups including dihedral groups admit Cayley graphs that are isomorphic to Cayley graphs of abelian groups, see~\cite{morris2020families}. Hence, in the bipartite case their sensitivity admits a lower bound in term of their degree by~\cite{potechin2020conjecture}.
Also, in~\cite[Remark 4]{verret2020counterexamples}, the authors describe their groups as close to abelian (dihedral groups have a cyclic group of index $2$, while $2$-groups are nilpotent). 
They ask for a natural family of Cayley graphs of non-abelian groups for which $\sigma$ grows in terms of the degree. 

To this end consider the following easy consequence of Huang's result. If a bipartite Cayley graph $G$ has a largest hypercube of dimension $\kappa(G)$ as a subgraph, then $\sigma(G)\geq\sqrt{\kappa(G)}$ (Proposition~\ref{pr:deltasubgraph})\footnote{Note that this observation is also essential for the result for abelian groups in~\cite{potechin2020conjecture}.}. In light of the second part of Huang's question it is thus natural to ask when this bound is tight. Clearly, all the three above families and also the family of~\cite{verret2020counterexamples} have $\kappa\equiv 1$ and hence they give tight examples for this bound. In~\cite{CFGS88}, Chung, F\"uredi, Graham, and Seymour show that Huang's bound is tight for the hypercube itself, i.e., $\sigma(Q_d)=\lceil\sqrt{d}\rceil$. We generalize this construction to sublattices of the hypercube (Lemma~\ref{lem:bound}).

We obtain infinite families of Cayley graphs with unbounded $\kappa$, where Huang's lower bound is tight. Namely, we study Coxeter groups. Our main result here is that the Cayley graph $G$ of a Coxeter group of type $\mathbf{A_n}$ or $\mathbf{I_2}(2k+1)$ as well as their direct products satisfy $\sigma(G)=\lceil\sqrt{\kappa(G)}\rceil$ (Corollary~\ref{cor:An}).  We furthermore extend this result to type $\mathbf{I_2}(n)\times \mathbf{I_2}(n')$ (Theorem~\ref{th:small}) as well as to many small Coxeter groups with the help of a computer (Table~\ref{tab:Q}). Moreover, we show that graphs $G$ of Coxeter groups of type $\mathbf{B_n}$ and $\mathbf{D_n}$ satisfy $\sigma(G)\leq\lceil\sqrt{\kappa(G)}\rceil+1$ (Theorem~\ref{th:coxeter}). 
We conjecture, that for every Cayley graph $G$ of a Coxeter group $\sigma(G)=\lceil\sqrt{\kappa(G)}\rceil$ (Conjecture~\ref{conjecture}).

Next, we study the sensitivity of bipartite Cayley graphs in the absence of cubes, i.e., where Proposition~\ref{pr:deltasubgraph} cannot be applied. We show that the Levi graphs of projective planes have unbounded sensitivity (Corollary~\ref{cor:pp}). Further we show that (Kronecker double covers of) the Ramanujan graphs of Lubotzky, Phillips, and Sarnak have unbounded sensitivity (Corollary~\ref{cor:Ramanujan}). Thus, providing families of cube-free, bipartite Cayley graphs that behave similarly to the hypercube with respect to sensitivity. The second family in particular has unbounded girth.

In the final section, after some concluding remarks we give an outlook on sensitivity in non-bipartite Cayley graphs. We show that the first guess on how to generalize the hypercube to higher chromatic number fails (Theorem~\ref{thm:Z3}).

Our experimental results were obtained combining SageMath~\cite{sage}, GAP~\cite{GAP4}, and CPLEX~\cite{cplex2009v12}.

\section{The dihedral group}

Let $D_n$ denote the dihedral group of symmetries of a regular $n$-gon, that is, the group \[ D_n = \langle a, b \, \vert \, a^n = b^2 = (ab)^2 = 1\rangle = \{1,a,\ldots,a^{n-1},b,ab,\ldots,a^{n-1}b\}.\]

 For a positive integer $m$, we denote by $[m]_3 \in \{1,2\}$ the right-most nonzero entry in its representation in base $3$. For example, for $m = 33$ we have that $m = 3^3 + 2 \cdot 3$ and, thus, $[m]_3 = 2$. 

The following result provides a family of bipartite $(d+1)$-regular dihedrants with sensitivity $1$ for all $d \geq 0$.

\begin{theorem} \label{thm:mathchindihedral} Let $n = 3^d$ and consider $G = {\rm Cay}(D_n, C)$, where $C = \{a^{3^i} b  \mid 0 \leq i \leq d\} \subseteq D_n$. The set $M = \{a^{i} \mid [i]_3 = 1\} \cup \{ a^{i} b \mid [i]_3 = 2\} \cup \{1,b\}$ induces a matching with $n+1$ vertices.  As a consequence, $\sigma(G) = 1$.
\end{theorem}
\begin{proof} Denote $c_{\ell} = a^{3^{\ell}} b$ for all $0 \leq \ell \leq d$. Take $x \in M$ and let us prove that it has exactly one neighbor in $M$. We separate the proof in four cases.

If $x = a^{i}$ with $[i]_3 = 1$. We take $j$ the largest exponent such that $3^j$ divides $i$ and write $i = \sum_{j<m<d} \beta_m 3^m + 3^j$. We observe that for all $\ell \in \{0,\ldots,d\}$ 
\[ [(i + 3^{\ell}) \mod n]_3 = \left\{ \begin{array}{lll} 1 & \text{ if } \ell \neq j, \\ 2 & \text{ if } \ell = j. \end{array} \right. \]
Hence  $x c_\ell = a^i a^{3^{\ell}} b = a^{i+3^{\ell}} b \in M$ if and only if $\ell = j$. As a consequence, $x$ has exactly one neighbor in $M$.

If $x = a^{i}b$ with $[i]_3 = 2$. We take $j$ the largest exponent such that $3^j$ divides $i$ and write $i = \sum_{j<m<d} \beta_m 3^m + 2 \cdot 3^j$. We observe that for all $\ell \in \{0,\ldots,d\}$ 
\[ [(i - 3^{\ell}) \mod n]_3 = \left\{ \begin{array}{lll} 1 & \text{ if } \ell = j, \\ 2 & \text{ if } \ell = j. \end{array} \right. \]
Hence  $x c_\ell = a^i b a^{3^{\ell}} b = a^{i-3^{\ell}} \in M$ if and only if $\ell = j$. As a consequence, $x$ has exactly one neighbor in $M$.

If $x = 1$, it is clear that $x c_\ell = c_\ell \in M$ if and only if $\ell = d$. 

If $x = b$, it is clear that $b c_\ell = a^{n - 3^\ell} \in M$ if and only if $\ell = d$. 

Hence $M$  induces a  matching and it is easy to check that $M$ has $n + 1$ elements. As a consequence $\sigma(G) = 1$. 
 \end{proof}

Exhaustive enumeration by computer shows that there is no smaller bipartite non-cyclic Cayley graph with $\sigma=1$, than the cubic $18$-vertex dihedrant given by Theorem~\ref{thm:mathchindihedral} for $n=9$. This graph has been obtained earlier by~\cite{verret2020counterexamples}.

\section{Star graphs}
The \emph{star graph} is the bipartite graph $SG_n=\Cay(S_n,\{(12),(13),\ldots, (1n)\})$. 
%
%
As the main result of this section, we will see in Theorem \ref{thm:stargraph} that star graphs all have sensitivity equal to 1. In other words, we will show that they have an induced subgraph with more than half of the vertices and maximum degree equal to 1.  

Given $\pi \in S_n$, we denote its support by ${\rm supp}(\pi) = \{i \in \{1,\ldots,n\} \, \vert \, \pi(i) \neq i\}$. A permutation $\pi \in S_n$ is called a \emph{derangement} if it has no fixed points or, in other words, if ${\rm supp}(\pi) = \{1,\ldots,n\}$. 

\begin{lemma}\label{lem:oddderangement}Let $n \in \mathbb{Z}^+$ and denote by $d_n$ the number of derangements of $n$ elements. Then, $d_n$ is odd if and only if $n$ is even.
\end{lemma}
\begin{proof} It is easy to check that $d_n$ satisfies the recursive formula $d_n = (n-1) (d_{n-1} + d_{n-2})$ for all $n \geq 3$. Since $d_1 = 0$ and $d_2 = 1$, the result follows by induction.
\end{proof}

\begin{lemma}\label{lem:neighbors} If $\pi, \tau \in S_n$ are adjacent in the star graph $SG_n$, then the sets ${\rm supp}(\pi) - \{1\}$ and ${\rm supp}(\tau) - \{1\}$ differ in at most one element. 
\end{lemma}
\begin{proof}Since $\pi$ and $\tau$ are adjacent in $SG_n$, then $\tau = \pi \cdot (1 r)$ for some $r \in \{2,\ldots,n\}$. We are going to prove that the symmetric difference of ${\rm supp}(\pi)$ and ${\rm supp}(\tau)$ is contained in $\{1,r\}$ and, hence, the result follows.
We write $\pi = c_1  \cdots c_t$ as a product of cycles with disjoint support, we clearly have that ${\rm supp}(\pi) = \cup_{i = 1}^t {\rm supp}(c_i)$. We divide the proof in several cases:

 If $1,r \notin {\rm supp}(\pi)$. Then $\tau = c_1  \cdots c_t \cdot (1 r)$ is a product of cycles with disjoint support, thus ${\rm supp}(\tau) = {\rm supp}(\pi) \cup \{1,r\}$.
 
  If $1 \notin {\rm supp}(\pi),\, r \in {\rm supp}(\pi)$. We may assume that $c_1 = (r  b_2  \cdots  b_k)$, then $\tau = (1 r  b_2 \cdots  b_k) \cdot c_2 \cdots c_t$ and, thus,  ${\rm supp}(\tau) = {\rm supp}(\pi) \cup \{1\}$.

  If $1 \in {\rm supp}(\pi),\, r \notin {\rm supp}(\pi)$. Proceeding as in the previous case we have that ${\rm supp}(\tau) = {\rm supp}(\pi) \cup \{r\}$.

  If $1,r \in {\rm supp}(\pi)$ and both belong to the support of different disjoint cycles, say  $c_1 = (1\ a_2 \ \cdots \ a_k)$, $c_2 = (r\ b_2 \ \cdots \ b_l)$. Then $c_1 \cdot c_2 \cdot (1 r) = (1 a_2  \cdots a_k r  b_2 \cdots b_l)$. Thus, ${\rm supp}(\tau) = {\rm supp}(\pi )$. 
  
  If $1,r \in {\rm supp}(\pi)$ and both are in the support of the same cycle, say $c_1 = (1 a_2 \cdots a_k)$ and $r = a_i$ for some $i \in \{2,\ldots,k\}$. If $k = 2$, then $c_1$ is the permutation $(1 r)$ and  ${\rm supp}(\tau)  = {\rm supp}(\pi) - \{1,r\}$. If $k > 2$ and $r = b_2$, then $c_1 \cdot (1 r) = (r a_3  \cdots a_k)$ and ${\rm supp}(\tau) = {\rm supp}(\pi) - \{1\}$. If $k > 2$ and $r = a_k$, then  $c_1 \cdot (1 r) = (1  a_2  \cdots a_{k-1})$ and  ${\rm supp}(\tau) = {\rm supp}(\pi) - \{r\}$. Finally, if $k > 2$ and $r = a_i$ with $2 < i < k$, then  $c_1 \cdot (1 r) = (1 a_2 \cdots a_{i-1}) \cdot (r a_{i+1} \cdots a_k)$ and ${\rm supp}(\tau) = {\rm supp}(\pi).$
\end{proof}

\begin{theorem}\label{thm:stargraph} The star graph $SG_n=\Cay(S_n,\{(12),(13),\ldots, (1n)\})$ has an induced subgraph with more than half of the vertices and maximum degree equal $1$. In other words, $\sigma(SG_n) = 1$.
\end{theorem}
\begin{proof}
Let $H$ be the \emph{domino}, that is, the graph with vertices $\{u_1,u_2,u_3,v_1,v_2,v_3\}$ and edges $\{u_1 u_2, u_2 u_3, v_1v_2, v_2v_3, u_1v_1, u_2v_2, u_3v_3\}$. 
Consider the map $f: V(SG_n) \longrightarrow V(H)$ defined as
\begin{center}
$f(\pi) = \left\{ \begin{array}{lll} u_1, &$ if $|\,{\rm supp}(\pi) - \{1\}| = n-1$ and $\pi \in A_n, \\
v_1, &$  if $|\,{\rm supp}(\pi) - \{1\}| = n-1$ and $\pi \notin A_n, \\
  u_2, & $ if $|\,{\rm supp}(\pi) - \{1\}| = n-2$ and $\pi \notin A_n, \\
    v_2, & $ if $|\,{\rm supp}(\pi) - \{1\}| = n-2$ and $\pi \in A_n, \\
     u_3, & $ if $|\,{\rm supp}(\pi) - \{1\}| < n-2$ and $\pi \in A_n, \\
      v_3, & $ if $|\,{\rm supp}(\pi) - \{1\}| < n-2$ and $\pi \notin A_n.
  \end{array} \right.$\end{center}
  
  \begin{figure}[ht]
\includegraphics[width=.9\textwidth]{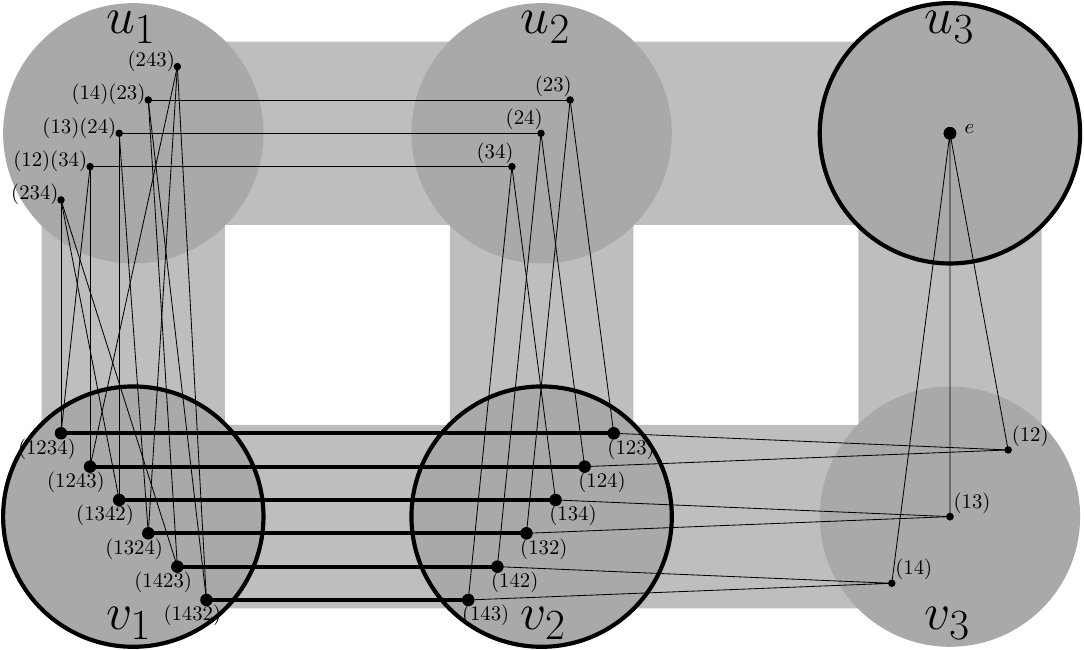}
\caption{Homomorphism from the Nauru graph $SG_4$ onto the domino.} \label{fig:homgraphstar}
\end{figure}

	  Let us check that $f$ is a graph homomorphism (see Figure \ref{fig:homgraphstar} for an example when $n = 4$). We observe that $f(A_n) = \{u_1,v_2,u_3\}$, $f(S_n - A_n) = \{v_1,u_2,v_3\}$. Since the domino is the complete bipartite graph $K_{3,3}$ minus the edges $u_1v_3$, $v_1u_3$, in order to prove that $f$ is a homomorphism we just have to check that if $f(\pi) = u_1$ and $f(\tau) = v_3$ (respectively, $f(\pi) = v_1$ and $f(\tau) = u_3$), then $\pi$ and $\tau$ are not neighbors in $SG_n$; this follows from Lemma \ref{lem:neighbors}. 
  
 Now we are going to prove that the induced subgraphs $K$ and $K'$ with vertices $f^{-1}(\{u_1,u_2,v_3\})$ and  $f^{-1}(\{v_1,v_2,u_3\})$, respectively, have both maximum degree equal to $1$.  Let $\pi \in V(K)$, we separate the proof in three cases:

\emph{Case $f(\pi) = v_3$.} Then $\pi$ has no neighbors in $K$ (since $f$ is a homomorphism). 

\emph{Case $f(\pi) = u_2$.} Then, $\pi \notin A_n$ and $|\,{\rm supp}(\pi) - \{1\}| = n-2$. Let $r$ be the only element in $\{2,\ldots,n\} - {\rm supp}(\pi)$. If $s \in \{2,\ldots,n\} - \{r\}$,  then $r$ is a fixed point for $\pi \cdot (1 s) \in A_n$ and then, $\pi \cdot (1 s) \notin V(K)$ because $f(\pi \cdot (1 s)) \in \{v_2,u_3\}$. As a consequence, the only neighbor of $\pi$ that might belong to $V(K)$ is $\pi \cdot (1 r)$ and the degree of $\pi$ in $K$ is at most $1$.
  
\emph{Case $f(\pi) = u_1$.} Then $\pi \in A_n$ and $|\,{\rm supp}(\pi) - \{1\}| = n-1$. We separate two cases:
\begin{itemize} \item If $1 \notin {\rm supp}(\pi)$. Then $\pi \cdot (1 s) \notin A_n$ is a derangement for all $s \in \{2,\ldots,n\}$. Therefore $\pi \cdot (1  s) \notin K$ because $f(\pi \cdot (1  s)) = v_1$. Thus, $\pi$ is an isolated vertex in $K$.
\item If $1 \in {\rm supp}(\pi)$. Let $r = \pi^{-1}(1) \in \{2,\ldots,n\}$. If $s \in \{2,\ldots,n\} - \{r\}$, then $\pi \cdot (1 s) \notin A_n$ and we claim that $\{2,\ldots,n\} \subseteq {\rm supp}(\pi \cdot (1 s))$. 
To prove the claim we take $i \in \{2,\ldots,n\}$ and we aim at proving that $[\pi \cdot (1 s)](i) \neq i$. We know that $\pi(i) \neq i$; we separate three cases:
 \begin{itemize}
  \item if $\pi(i) \notin \{1,s\}$, then $[\pi \cdot (1 s)](i) = \pi(i) \neq i$,
 \item  if $\pi(i) = 1$, then $i = r$ and $[\pi \cdot (1 s)](i) = s \neq r = i$; and
 \item if $\pi(i) = s$, then $[\pi \cdot (1 s)](i) = 1 \neq i$.
 \end{itemize}
Thus, we conclude that $f(\pi \cdot (1 s)) = v_1$ and $\pi \cdot (1 s) \notin K$ for all $s \neq r$. As a consequence, the only neighbor of $\pi$ that might belong to $V(K)$ is $\pi \cdot (1 r)$ and the degree of $\pi$ in $K$ is at most $1$.
  \end{itemize}
  
  A similar argument works for $K'$. To get the result we now prove that $K$ and $K'$ do not have the same number of elements and, as a consequence, one has more than half of the vertices of $SG_n$ (see Figure \ref{fig:homgraphstar} for the case $n = 4$, where $K'$ has 13 vertices).  Since $f^{-1}(\{u_1,v_2,u_3\}) = A_n$ and $f^{-1}(\{v_1,u_2,v_3\}) = S_n - A_n$ and both sets have the same cardinality, we just need to verify that $f^{-1}(u_1)$ and $f^{-1}(v_1)$ do not have the same number of elements. It suffices to observe that the elements of $f^{-1}(\{u_1,v_1\})$ are in bijection with the set of derangements of either $n$ or $n-1$ elements and, thus,  $|f^{-1}(u_1)| + |f^{-1}(v_1)| = d_n + d_{n-1}$ which, by Lemma \ref{lem:oddderangement}, is an odd number. This completes the proof.  
\end{proof}


One can be more precise in the proof of Theorem \ref{thm:stargraph} and determine that $K$ has exactly $\frac{n!}{2} + (-1)^{n+1}$ vertices and $K'$ has  $\frac{n!}{2} + (-1)^{n}$ vertices. Indeed, following the notation of the proof, we have that $|f^{-1}(u_1)|$ equals the number of even (belonging to $A_n$) derangements of $n$ elements plus the  number of even derangements of $n-1$ elements, then by~\cite[sequence A003221]{EOIS}  we have that 
\[ |f^{-1}(u_1)| = \frac{d_n - (-1)^n(n-1)}{2} + \frac{d_{n-1} - (-1)^{n-1}(n-2)}{2} = \frac{d_n + d_{n-1} + (-1)^{n+1}}{2}\]
and $|f^{-1}(v_1)| = \frac{d_n + d_{n-1} - (-1)^{n+1}}{2} =  |f^{-1}(u_1)| - (-1)^{n+1}$ and we get that 
\[ |V(K)| =  |A_n| + |f^{-1}(u_1)| - |f^{-1}(v_1)| = \frac{n!}{2} + (-1)^{n+1}.\]
Thus, we conclude that the graph with more that half of the vertices of $SG_n$ is $K$ for $n$ odd, and $K'$ for $n$ even.

\section{Tight groups}\label{sec:tight}
It is easy to see that an induced subgraph of maximum degree $1$ in a $d$-regular $n$-vertex graph has at most $\frac{d}{2d-1}n$ vertices. We say that a graph is \emph{tight} if it attains equality. Lehner and Verret ask if there are tight Cayley graphs of groups, see~\cite[Remark 2]{verret2020counterexamples}. Here we give some examples and an infinite family.

First of all one has that $\Cay(D_{3m},\{b,ab\})\cong C_{6m},$ the cycle graph on $6m$ vertices. This graph has an induced matching on $\frac{2}{3}$ of the vertices, hence it is tight of degree $2$.

An exhaustive computer search yields that on up to $60$ vertices there are exactly three tight cubic Cayley graphs. Two of them on $50$ and $60$ vertices, respectively, are depicted in Figure~\ref{fig:A5}. The other one is another Cayley graph of $A_5$ and is the first member of the infinite family shown in Theorem~\ref{thm:alternatingtight}.

  \begin{figure}[ht]
\includegraphics[width=\textwidth]{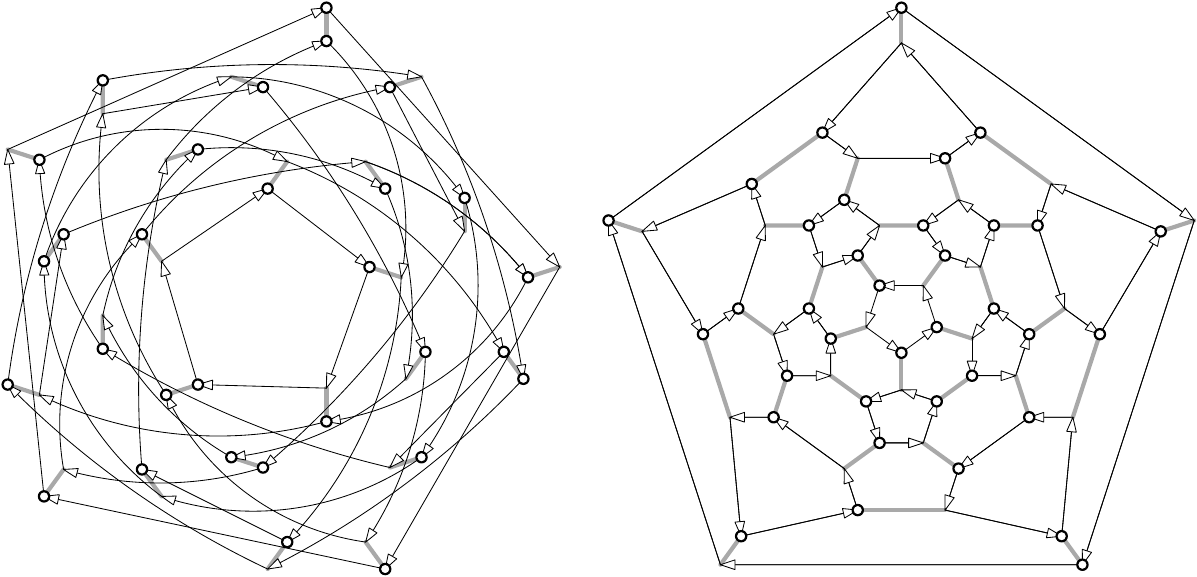}
\caption{The smallest cubic tight Cayley graphs $\Cay(\mathbb{Z}_5\times D_5,\{(1,a),(0,b)\})$ and $\Cay(A_5,\{(12345),(12)(34))\})$, which is an orientation of the skeleton of the truncated icosahedron. The white vertices induce matchings on $\frac{3}{5}$ of the vertices.} \label{fig:A5}
\end{figure}
%
%

\begin{theorem}\label{thm:alternatingtight}
For $m \in \mathbb{Z}^+$, let $\Gamma = S_{2m+1}$ if $m$ is odd and $\Gamma = A_{2m+1}$ if $m$ is even. Further, let $c_k$ be the order $2$ 
permutation of $\{1,\ldots,2m+1\}$ defined by:
\[ c_k(i) = \left\{ \begin{array}{llll} i + m & \text{if} & i < k-m, \\  i + m + 1 & \text{if} & k-m \leq i \leq m, \\ 
i - m & \text{if} & m < i < k \\ i & \text{if} & i = k \\ i - m - 1 & \text{if} & k < i \leq 2m+1.  \end{array} \right.  \]
Then, the set $M = \{\pi \in \Gamma \, \vert \, \pi(1) \geq m+1\}$ has $\frac{m+1}{2m+1}|\Gamma|$ elements and induces a matching in $G = \Cay(\Gamma, C)$ with $C = \{ c_k \, \vert \, m+1 \leq k \leq 2m+1\}$.
\end{theorem}
\begin{proof}We observe that the signature of $c_k$ is $(-1)^m$ and then $c_k \in A_{2m+1}$ if and only if $m$ is even. Now, we consider the partition  $M =\sqcup_{i = m+1}^{2m+1} M_i$, where $M_i = \{\pi \in M \, \vert \, \pi(1) = i\}$. It is clear that $|M_i| = |\Gamma|/(2m+1)$ for all $i$, and then $M$ has $\frac{m+1}{2m+1}|\Gamma|$ elements. 

Take $i \in \{m+1,\ldots,2m+1\}$ and consider $\pi \in M_i$. We claim that $\pi \cdot c_k \in M$ if and only if $k = i$ and, as a consequence, $M$ induces a matching in $G$. Indeed, $\pi \cdot c_i \in M_i \subseteq M$ because $[\pi \cdot c_i] (1) = c_i(\pi(1)) = c_i(i) = i$ and, for all $k \neq i$, then $\pi \cdot c_k \notin M$ because $[\pi \cdot c_k](1) = c_k(\pi(1)) = c_k(i) \leq m$
\end{proof}

 We wonder if the set $C$ described in this result is a minimal set of generators of $\Gamma$ in every case. Otherwise, the subgroup of $\Gamma$ spanned by $C$ would provide a smaller tight group. 

It is also worth pointing out that the same result (and the same argument of the proof) holds for any set $C = \{ c_k \, \vert \, m+1 \leq k \leq 2m+1\} \subseteq \Gamma$ satisfying that $c_k$ is an order $2$ permutation with $c_k(k) = k$ and $c_k(i) \leq m$ for all $k \neq i \geq m+1$.

%
%
%
We remark that while the above construction gives a tight Cayley graph for every degree, the obtained graphs are pretty large. E.g., for degree $4$, we obtain a Cayley graph of $S_7$. However, we know of at least one smaller such graph, namely $\Cay(A_7,\{(1234567),(123)(45)(67)\})$. It has degree 4 and 2520 vertices and an induced matching of 1440 vertices, i.e., it is tight.  We wonder what size the smallest $4$-regular tight Cayley graph is. By computational means we checked that the answer is at least $84$.

Note further, that the above graphs are the only non-bipartite graphs that have appeared so far. However, we can also construct bipartite ones. For this we recall a couple of definitions and prove a lemma that has been used implicitly in~\cite{potechin2020conjecture,verret2020counterexamples}. A \emph{covering map} from a graph $\hat{G}$ to a graph $G$ is a surjective graph homomorphism $\varphi:\hat{G}\to G$ such that for every vertex $v\in \hat{G}$,  $\varphi$ induces a  one-to-one correspondence between edges incident to $v$ and edges incident to $\varphi(v)$. If there is a covering map from $\hat{G}$ to $G$, we say that $\hat{G}$ is a \emph{covering} of $G$. Finally, for a graph $G$ and every $0 < \beta \leq 1$, we denote by $\Delta_{\beta}(G)$ the minimum value $\Delta(G[H])$ among all the  $H \subset V(G)$ with $|H| \geq \beta |V(G)|$. 

\begin{lemma}\label{lem:covering}
 Let $\hat{G}$ be a {covering} of $G$ and $\beta \in (0,1]$. Then $ \Delta_{\beta}(\hat{G}) \leq \Delta_{\beta}(G)$.
\end{lemma}
\begin{proof} Let $\varphi:\hat{G}\to G$ be a covering map and let us assume without loss of generality that $G$ is connected. It is easy to see that all fibers of $\varphi$ have the same size $k$ and, thus $|V(\hat{G})| = k \cdot |V(G)|$. Now, take $K \subset V(G)$ such that $|K| \geq \beta |V(G)|$ and $\Delta_{\beta}(G) = \Delta(G[K])$. Considering $\hat{K} := \varphi^{-1}(K)$ one has that $|\hat{K}| = k \cdot|K|$ and then $\frac{|\hat{K}|}{|V(\hat{G})|} = \frac{|K|}{|V(G)|} \geq \beta$. Since $\varphi$ is a homomorphism and two neighbors of a given vertex cannot be mapped by $\varphi$ to the same vertex, then the maximum degree induced by $\hat{K}$ is at most the maximum degree induced by $K$. This yields the claim.
\end{proof}

The \emph{cross product} $G\times H=(V\times V',E'')$ of two graph $G=(V,E)$ and $H=(V',E')$ has an edge $\{(u,u'),(v,v')\}\in E''$ if and only if $\{u,v\}\in E$ and $\{u',v'\}\in E'$. The \emph{Kronecker double cover} of a graph $G$ is the bipartite graph $G \times K_2$. It is easy to see that $G \times K_2$ is a covering of $G$.
\begin{remark}\label{rm:cayleydouble} Given a Cayley graph $\Cay(\Gamma,C)$ its Kronecker double cover $\Cay(\Gamma,C)\times K_2$ is the bipartite Cayley graph $\Cay(\Gamma\times\mathbb{Z}_2,C\times\{1\})$. 
\end{remark}

This remark together with Lemma~\ref{lem:covering} and Theorem~\ref{thm:alternatingtight} yield:
\begin{corollary}\label{cor:bipdouble}
 There are infinite families of unbounded degree bipartite tight Cayley graphs.
\end{corollary}

We have checked with a computer that the smallest cubic bipartite tight Cayley graph comes from the above construction and is a Cayley graph of $\mathbb{Z}_{10}\times D_5$. We do not know which is the smallest $4$-regular bipartite tight Cayley graph.


Note that a source of tight transitive graphs are \emph{odd graphs}, see~\cite{verret2020counterexamples}. In particular, the smallest cubic tight transitive graph is the \emph{Petersen graph} $G(5,2)$ and the smallest cubic bipartite tight transitive graph is its Kronecker cover, namely the \emph{Desargues graph} $G(10,3)$.

%

\section{Bounds and constructions close to the hypercube}
In the present section we give a very elementary generalization of the lower bound of Huang~\cite{huang2019induced} and a more involved generalization of the construction of Chung, F\"uredi, Graham, and Seymour~\cite{CFGS88}. Both will be applied in the following section to Coxeter groups.

For any graph $G$, we denote by \[ \kappa(G) = \max\{n \in \mathbb{Z}^+ \, \vert \, Q_n \text{ is a subgraph  of }G\}, \]  i.e., $\kappa(G)$ is the dimension of the largest hypercube contained in $G$.

\begin{proposition}\label{pr:deltasubgraph}Let $G$ be a bipartite Cayley graph and $H$ a regular subgraph of $G$, then $\sigma(G)\geq\sigma(H)$. In particular,  $\sigma(G) \geq \sqrt{\kappa(G)}$.
\end{proposition}
\begin{proof} Since $G$ and $H$ are bipartite and regular, their maximum independent sets contain half the vertices. Now, for every $x \in V(G)$, we consider the set of vertices $x \cdot H = \{x \cdot h \, \vert \, h \in V(H)\}$. The sets $(x \cdot H)_{x \in V(G)}$ cover $G$ and every element in $G$ belongs to exactly $|V(H)|$ of these sets.  If one takes a set $K \subseteq V(G)$ with $|K| > \frac{1}{2} |V(G)|$, then 
\[ \sum_{x \in V(G)} |K \cap (x \cdot H)| = |K|  |V(H)| > \frac{1}{2}  |V(G)| |V(H)|\] and, by the pigeonhole principle,
there exists an $x \in V(G)$ such that $|K \cap (x \cdot H)| > \frac{1}{2}  |V(H)|$. Since the induced graph with vertices $x \cdot H$ is isomorphic to $H$, we conclude that the maximum  degree of the subgraph induced by $K$ is at least $\sigma(H)$. The second statement follows from Huang's result~\cite{huang2019induced}.
\end{proof}

\bigskip
Before we go into constructions let us introduce a coloring variant of the parameter $\sigma$:  For a graph $G=(V,E)$ and a non-negative integer $k$ denote by $$\iota_k(G)=\max\{|A|-|B| \mid V=A\sqcup B\text{ and }\Delta(G[A]),\Delta(G[B])\leq k\}$$ its \emph{k-imbalance}. Hence, for a regular bipartite graph $G$ there is a subset $K$ on $\frac{|V|+\iota_k(G)}{2}$ vertices with $\Delta(G[K])\leq k$. In particular, $\sigma(G)\leq\min\{k\mid \iota_k(G)>0\}$. An easy observation is that if $H$ is a subgraph of $G$ with the same vertex set, then $\iota_k(G)\leq \iota_k(H)$ for every $k\geq 0$. For the next property, define the \emph{Cartesian product} of graphs $G=(V,E)$ and $H=(V',E')$ as $G\square H=(V\times V',E'')$, where $\{(u,u'),(v,v')\}\in E''$ if and only if $u=v$ and $\{u',v'\}\in E'$ or $u'=v'$ and $\{u,v\}\in E$.

\begin{lemma}\label{lem:prodimb}
 For graphs $G,H$ and non-negative integers $k,\ell$, we have $\iota_k(G)\iota_{\ell}(H)\leq \iota_{k+\ell}(G\square H)$.
\end{lemma}
    \begin{proof}
Let $A\sqcup B$ be a partition of $G$ such that both sets induce subgraphs of maximum degree at most $k$ and $|A|-|B|=\iota_k(G)$. Similarly, let $A'\sqcup B'$ be a partition of $H$ such that both sets induce subgraphs of maximum degree at most $\ell$ and $|A'|-|B'|=\iota_{\ell}(H)$.
 
Define two new sets $A''=A\times A'\cup B\times B'$ and $B''=A\times B'\cup B\times A'$. Clearly, $A''$ and $B''$ partition the vertex set of $G\square H$. Let us analyze without loss of generality the maximum degree induced by $A''$. Let $v= (a,a')$ with $a\in A\subseteq V(G)$ and $a' \in A'\subseteq V(H)$. The degree of $v$ is constituted by its degree in $A''\cap \{a\}\times A'$ and its degree in $A''\cap A\times \{a'\}$. Thus, it equals the sum of the degree of $a$ in $A$ and the degree of $a'$ in $A'$.  The analogous argument holds for $v=(b,b')$ with $b\in B\subseteq V(G)$ and $b' \in B'\subseteq V(H)$. We conclude that both $A''$ and $B''$ induce subgraphs of maximum degree at most $k+\ell$.

Finally, we compute the imbalance $\iota_{k+\ell}(G\square H)\geq |A''|-|B''|=|A||A'|+|B||B'|-|A||B'| - |B||A'|=(|A|-|B|)(|A'|-|B'|)=\iota_k(G)\iota_{\ell}(H)$.
\end{proof}

In \cite{CFGS88}, Chung, F\"uredi, Graham, and Seymour exhibited an induced subgraph of $Q_n$ with $2^{n-1}+1$ vertices and maximum degree $\left\lceil \sqrt{n} \right\rceil$ for all $n \geq 1$. Next we extend this construction to certain lattices.

We introduce some notation for posets and lattices. For a poset $P$, we say that $y$ \emph{covers} $x$ and we write $x \prec y$, if $x<y$ and there is is no $z\in P$ with $x<z<y$. We denote by $G_P=(P,E)$ its \emph{cover graph}, i.e, $\{x,y\}\in E$ whenever $x \prec y$. We say that $P\subseteq Q$ are \emph{cover subposets} if $x\leq_Py\iff x\leq_Qy$ for all $x,y\in P$ and $G_P$ is an induced subgraph of $G_Q$. For $x\in P$ denote by $\uparr x=\{y \in P\mid x\leq y\}$, and for $\mathbf{F} \subseteq P$ denote by $\uparr \mathbf{F} = \cup_{x \in F} \uparr x$.
A \emph{lattice} $\mathcal{L}$ is a partially ordered set, such that for any $x,y\in \mathcal{L}$ there a unique smallest element $x\vee y\geq x,y$ called the \emph{join} of $x$ and $y$ and a unique largest element $x\wedge y\leq x,y$ called the \emph{meet} of $x$ and $y$. The $\emph{Boolean lattice}$ $\mathcal{B}_n$ is the inclusion order of all subsets of the set $[n]=\{1, \ldots, n\}$. Its cover graph is the hypercube $Q_n$.

We from now on consider a lattice $\mathcal{L}$ that is a cover subposet of $\mathcal{B}_n$. Before proceeding to studying sensitivity related results, let us discuss the generality of this class. First, note that $\mathcal{L}$ is not a necessary a sublattice of $\mathcal{B}_n$, i.e., it may have different join and meet operations. However, since $\mathcal{L}$ is a subposet of $\mathcal{B}_n$ we can assume without loss of generality that the minimum and maximum $\hat{0}$, $\hat{1}$ of $\mathcal{L}$ correspond to the empty and the full set in $\mathcal{B}_n$, respectively. See Figure~\ref{fig:lattices} for three examples.

\begin{figure}[ht]
\includegraphics[width=\textwidth]{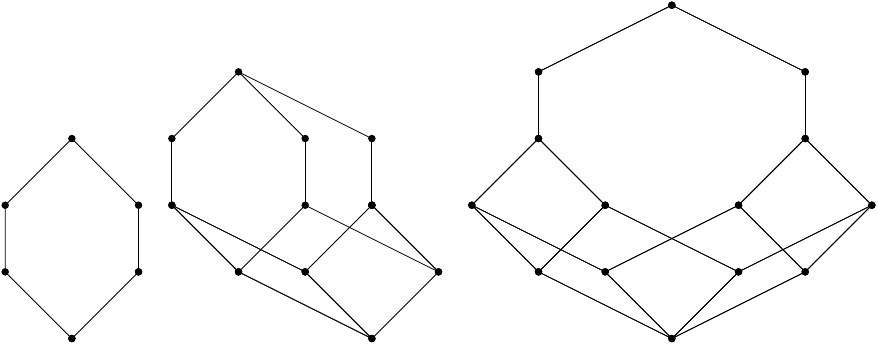}
\caption{Three lattices that are cover subposets of $\mathcal{B}_3$, $\mathcal{B}_4$, and $\mathcal{B}_5$, respectively.} \label{fig:lattices}
\end{figure}

We proceed to define an important subclass of these lattices.
A graph $G$ is a \emph{partial cube} if $G$ is (isomorphic to) an \emph{isometric subgraph} of ${Q}_n$, i.e., $d_G(x,y)=d_{Q_n}(x,y)$ for all $x,y\in G$, where $d$ denotes the distance function. Fixing $z\in G$ and defining $x\leq y$ if $d_G(z,y)=d_G(z,x)+d_G(x,y)$ yields a poset denoted $P(G,z)$.

\begin{lemma}\label{lem:partialcubes}
 If $G$ is a partial cube and $z\in G$ a vertex, then the poset  $P(G,z)$ is isomorphic to a cover subposet of $\mathcal{B}_n$ with cover graph $G$.
\end{lemma}
\begin{proof}
 Let $G$ be (isomorphic to) an isometric subgraph of $Q_n$ and choose the isomorphism such that $z$ is identified with the empty set in $\mathcal{B}_n$. Since $G$ is bipartite, $G$ is the cover graph of $P(G,z)$. Furthermore, since $G$ is isometric it is in particular an induced subgraph of $Q_n$. 
 Let, now $x\leq_{P(G,z)}y$ in $P(G,z)$. This by definition means $d_G(z,y)=d_G(z,x)+d_G(x,y)$ which by isometry condition is equivalent to $d_{Q_n}(z,y)=d_{Q_n}(z,x)+d_{Q_n}(x,y)$. Now, since $z$ corresponds to the empty set for the sets $X,Y$ corresponding to $x,y$ this means $X\subseteq Y$, which is equivalent to $x\leq_{\mathcal{B}_n}y$.
\end{proof}

Lemma~\ref{lem:partialcubes} yields a rich class of lattices that are cover subposets of a Boolean lattice:
\begin{remark}\label{rem:lattice}
If $G$ is a partial cube with a vertex $z\in G$ such that $P(G,z)$ is a lattice $\mathcal{L}$, then $\mathcal{L}$ is a cover subposet of $\mathcal{B}_n$. The dual graph $G$ of a central hyperplane arrangement is a partial cube, see e.g.~\cite{LV80,E06}. If the hyperplane arrangement is simplicial, then $G$ is regular and for any vertex $z\in G$ the poset $P(G,z)$ is a lattice $\mathcal{L}$ and $G=G_{\mathcal{L}}$, see~\cite{BEZ90}.
\end{remark}

The left-most lattice in Figure~\ref{fig:lattices} arises from a central hyperplane as described in Remark~\ref{rem:lattice}. Indeed the so-called \emph{weak (right) order} of a Coxeter group~\cite{B84} is an example, that arises from taking the dual graph of a Coxeter arrangement. See the left of Figure~\ref{fig:B3} for another example.
The lattice in the middle of Figure~\ref{fig:lattices} arises from a partial cube, that is not the dual graph of a hyperplane arrangement. The right-most lattice in Figure~\ref{fig:lattices} arises from an induced subgraph of $Q_5$, that is not a partial cube.

We return to studying sensitivity related notions. In a lattice $\mathcal{L}$ that is a cover subposet of $\mathcal{B}_n$,
we call the vertices \emph{even} and \emph{odd} depending on the cardinalities of the corresponding sets. The set of even and odd vertices of a subset $S\subseteq \mathcal{L}$ is denoted $\even(S)$ and $\odd(S)$, respectively. For $\mathbf{F}\subseteq \mathcal{L}$ define $r(\mathbf{F})=\max\{|F| \mid F\in \mathbf{F}\}$ and $t(\mathbf{F})=\max\{|X| \mid X \subseteq\mathbf{F} {\rm\ and \ } \forall F\in X:F\setminus( \bigcup_{K \in X \atop K \neq F} K) \neq \emptyset\}$. This is, $t(\mathbf{F})$ denotes the size of a largest subset $X$ of $\mathbf{F}$ such that every $F\in X$ contains an element that is in no other set from $X$.
Given $\mathbf{F}$ we define:

$$\mathbf{X}(\mathbf{F})=\even(\uparr\mathbf{F})\cup\odd(\mathcal{L}\setminus\uparr\mathbf{F}).$$ 
Define $G(\mathbf{F}):=G_{\mathcal{L}}[\mathbf{X}(\mathbf{F})]$ and $G'(\mathbf{F}):=G_{\mathcal{L}}[\mathcal{L}\setminus\mathbf{X}(\mathbf{F})]$ as the induced subgraph of $G_{\mathcal{L}}$ on $\mathbf{X}(\mathbf{F})$ and on the complement of $\mathbf{X}(\mathbf{F})$, respectively.

\begin{lemma}\label{lem:bound}
Let $\mathcal{L}$ be lattice that is a cover subposet of $\mathcal{B}_n$, $\mathbf{F}\subseteq \mathcal{L}$, and $k=\max\{r(\mathbf{F}),t(\mathbf{F})\}$. We have: $$\max\{\Delta(G(\mathbf{F})),\Delta(G'(\mathbf{F}))\}\leq k.$$
As a consequence, if $G_{\mathcal{L}}$ is regular, then $$\frac{\iota_k(G_{\mathcal{L}})}{2}\geq\sum_{1\leq i\leq k}|\even(\uparr F_i)|-|\odd(\uparr F_i)|-\sum_{1\leq i<j\leq k}|\even(\uparr(F_i\vee F_j))|-|\odd(\uparr(F_i\vee F_j))|\pm\ldots$$
\end{lemma}
\begin{proof}
 Since the statement for $G'(\mathbf{F})$ is proved analogously, here we only prove $\Delta(G(\mathbf{F}))\leq\max\{r(\mathbf{F}),t(\mathbf{F})\}$.
 So let $\{S,S'\}$ be an edge of $G(\mathbf{F})$. 

 If $S$ is even, then $S'\prec S$ is a cover relation, $S'$ is odd, and for all $F\in\dwnarr(S)\cap\mathbf{F}$ we have $S'\vee F=S$. Thus, the coordinate corresponding to the element $S\setminus S'$ is contained in $\bigcap (\dwnarr(S)\cap\mathbf{F})$. Hence, $\deg(S)\leq |\bigcap (\dwnarr(S)\cap\mathbf{F})|\leq r(\mathbf{F})$.
 
 If $S$ is odd, then $S\prec S'$ is a cover relation, $S'$ is even and the only element $s \in S'\setminus S$ belongs to some $F\subseteq S'=S \cup \{s\}$ such that $F \in \mathbf{F}$. Thus, the neighbors $S\prec S'_1, \ldots S'_k$ give rise to a set $X=\{F_1, \ldots, F_k\}\subseteq\mathbf{F}$ such that $s_i\in F_i\setminus( \cup_{j \neq i} F_j)$ for all $1\leq i\leq k$. Thus, $\deg(S)\leq t(\mathbf{F})$.


 For the second part of the statement we estimate $|\mathbf{X}(\mathbf{F})|=|\even(\uparr \mathbf{F})|+|\odd(\mathcal{L}\setminus\uparr\mathbf{F})|$ via inclusion-exclusion. The size of the first term can be written as $$|\even(\uparr \mathbf{F})|=\sum_{1\leq i\leq k}|\even(\uparr F_i)|-\sum_{1\leq i<j\leq k}|\even(\uparr(F_i\vee F_j))|\pm\ldots$$

Similarly, we can express the size of the second term as:
$$|\odd(\mathcal{L}\setminus\uparr\mathbf{F})|=|\odd(\mathcal{L})|-\sum_{1\leq i\leq k}|\odd(\uparr F_i)|+\sum_{1\leq i<j\leq k}|\odd(\uparr(F_i\vee F_j))|\mp\ldots$$

Since $G_{\mathcal{L}}$ is bipartite and regular we have $|\odd(\mathcal{L})|=|\even(\mathcal{L})|=\frac{|\mathcal{L}|}{2}$. We can write $|\mathbf{X}(\mathbf{F})|$ as:
$$\frac{|\mathcal{L}|}{2}+\sum_{1\leq i\leq k}|\even(\uparr F_i)|-|\odd(\uparr F_i)|-\sum_{1\leq i<j\leq k}|\even(\uparr(F_i\vee F_j))|-|\odd(\uparr(F_i\vee F_j))|\pm\ldots$$
This concludes the proof.
\end{proof}
 

When $\mathcal{L}$ is the Boolean lattice $\mathcal B_d$ itself, the first part of Lemma~\ref{lem:bound} is \cite[Proposition 3.3]{CFGS88}. Thus, applying Lemma~\ref{lem:bound} with an appropriate set $\mathbf{F} = \{F_1,\ldots,F_k\}$; for example, $\mathbf{F}$ is any partition of $\{1,\ldots,d\}$ with $\sqrt{d} - 1 < k < \sqrt{d} + 1$ and $\sqrt{d}-1 < |F_i| < \sqrt{d}+1$ for all $i \in \{1,\ldots,k\}$, one recovers the following:

\begin{theorem}[\!\cite{CFGS88}]\label{thm:CFGS88}
 For any integer $d$, we have $\iota_{\lceil\sqrt{d}\rceil}(Q_d)\geq 2$. In particular, $\sigma(Q_d)\leq \lceil\sqrt{d}\rceil$.
\end{theorem}

Note that with Remark~\ref{rem:lattice} there is a wider class of lattices where Lemma~\ref{lem:bound} can be applied. In particular if the graph $G_{\mathcal{L}}$ of the lattice $\mathcal{L}$ is the dual graph of a simplicial hyperplane arrangement, then $G_{\mathcal{L}}$ is regular. Clearly, Lemma~\ref{lem:bound} is only useful together with a smartly chosen set $\mathbf{F}$. We will come back to this in the next section. 

\section{Coxeter groups}\label{sec:Coxeter}
We consider Cayley graphs of Coxeter groups and  provide explicit constructions showing that the bound in Proposition \ref{pr:deltasubgraph} in each case is either an equality or at most one unit away from an equality.

More precisely, we first introduce the notion of cube-like Coxeter groups. This allows us to establish equality for Coxeter groups of types $\mathbf{I_2}(2k+1)$, $\mathbf{A_n}$, and their products. Further we show equality for types $\mathbf{I_2}(n)$ and $\mathbf{I_2}(n)\times \mathbf{I_2}(n')$, and many small Coxeter groups by computer. We also show that types $\mathbf{B_n}$ and $\mathbf{D_n}$ cannot deviate by more than one unit from the lower bound.  We finish the section with a conjecture. 

We start with the necessary definitions and refer the reader to \cite{CombCoxeterBook, SimpleGroup} for more thorough introductions into the combinatorics of Coxeter groups. A finite Coxeter system is a pair $(W,S)$, where $W$ is a group with generators $S = \{a_1,\ldots,a_n\}$ and presentation $W = \langle a_1,\ldots, a_n \, \vert \, (a_i a_j)^{m_{ij}} = 1 \rangle$ where $m_{ij} > 1$ and $m_{ii} = 2$. In \cite{Coxeter}, Coxeter classified all finite Coxeter groups as (direct products of) the members of three infinite families of increasing rank $\mathbf{A_{n}},\mathbf{B_{n}},\mathbf{D_{n}}$, one  family of dimension two $\mathbf{I_2}(n)$, and six exceptional groups: $\mathbf{E_{6}},\ \mathbf{E_{7}},$ $\ \mathbf{E_{8}},\ \mathbf{F_{4}},\ \mathbf{H_{3}}$ and $\mathbf{H_{4}}$.

  \begin{figure}[ht]
\includegraphics[width=.9\textwidth]{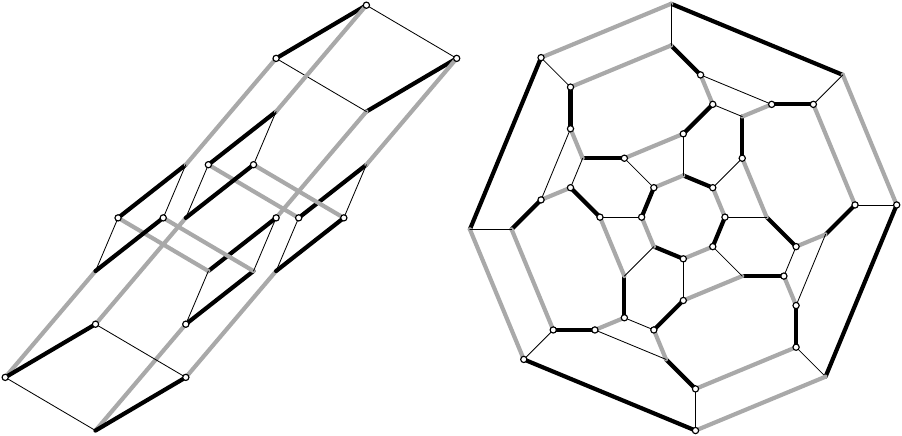}
\caption{Induced subgraphs of maximum degree $2$ in $\Cay(\mathbf{A_3})$ and $\Cay(\mathbf{B_3})$.} \label{fig:B3}
\end{figure}

Since any Coxeter group $W$ corresponds to a unique Coxeter system $(W,S)$, we denote the Cayley graph $\Cay(W,S)$ just as $\Cay(W)$.  See Figure~\ref{fig:B3} for a drawing of the Cayley graphs of $\mathbf{A_3}$ and $\mathbf{B_3}$. 
It is well known that the Cayley graph  of a Coxeter groups $\Cay(W)$ is the dual graph of a simplicial hyperplane arrangement -- the Coxeter arrangement of the corresponding type. As explained in Remark~\ref{rem:lattice} this gives that $\Cay(W)$ is a partial cubes -- an isometric subgraph of a hypercube. The dimension of this hypercube corresponds to the number of hyperplanes in the arrangement, which is the number $r$ of \emph{reflections} of $W$, i.e., the elements of order $2$. See Table~\ref{tab:K} for the number of reflections of the irreducible Coxeter groups and for the dimension of the largest hypercube contained in their corresponding Cayley graphs, this number coincides with the size of the largest independent set of their Coxeter-Dynkin diagrams. 

\begin{center}
\begin{table}[h]
\begin{tabular}{r|cccccccccc} 
 
 ~&$\mathbf{A_{n}}$&$\mathbf{B_{n}}$&$\mathbf{D_{n}}$&$\mathbf{I_2}(n)$&$\mathbf{E_{6}}$&$\mathbf{E_{7}}$&$\mathbf{E_{8}}$&$\mathbf{F_{4}}$&$\mathbf{H_{3}}$&$\mathbf{H_{4}}$\\
 \hline
 $\kappa$&$\lceil\frac{n}{2}\rceil$&$\lceil\frac{n}{2}\rceil$&$\lceil\frac{n+1}{2}\rceil$&$1$&$3$&$4$&$4$&$2$&$2$&$2$\\
 $r$&$\frac{n(n+1)}{2}$&$n^2$&$n(n-1)$&$n$&$36$&$63$&$120$&$12$&$10$&$30$\\

\end{tabular}
\caption{Largest cube and number of reflections in irreducible Coxeter groups. }\label{tab:K}
\end{table}
\end{center}

Following Remark~\ref{rem:lattice}, from the simpliciality of the Coxeter arrangement we get that $\Cay(W)$ is regular and the cover graph of a lattice $\mathcal{L}_W$ that is the cover subposet of a Boolean lattice. We are thus in the position to apply Lemma~\ref{lem:bound} once we have found an interesting set $\mathbf{F}$. In order to get there, we will proceed to introduce more specific properties of $\mathcal{L}_W$, mostly taken from~\cite{CombCoxeterBook,SimpleGroup}.
Taking as base-point of $\Cay(W)$ the neutral element $e\in W$, with the notation of Lemma~\ref{lem:partialcubes}, the lattice $P(\Cay(W),e)=\mathcal{L}_W$ is called the \emph{weak (right) order}~\cite{B84}. For two group elements we have $w\leq w'$ if $d_{\Cay(W)}(e,w')=d_{\Cay(W)}(e,w)+d_{\Cay(W)}(w,w')$. This makes it convenient to denote the \emph{length} of an element $w\in W$ is $\ell(w)$, which is the distance from $e$ in $\Cay(W)$, i.e., $\ell(w)=d_{\Cay(W)}(e,w)$.

For $J\subseteq S$, we denote by $W_J$ the subgroup of $W$ generated by $J$. The Coxeter system $(W_J,J)$ is called a \emph{parabolic subgroup} of $(W,S)$.
Note that the graph $\Cay(W_J)$ is a subgraph of $\Cay(W)$ and hence $\sigma(\Cay(W_J))\leq \sigma(\Cay(W))$, by Proposition~\ref{pr:deltasubgraph}.
The set $W^J=\{w\in W\mid \ell(wj)>\ell(w) \text{ for all }j\in J\}$ is the corresponding \emph{quotient}. We collect some facts about $W_J$ and $W^J$ with respect to $\mathcal{L}_W$.
\begin{enumerate}
 \item the elements of $W_J$ define an order-interval $I(W_J)$ that  induces a sublattice of $\mathcal{L}_{W}$,
 \item the elements of $W^J$ define an order-interval $I(W^J)$, whose graph we denote by $G(W^J)$, moreover $I(W^J)$ is isomorphic to the reversed order $I(W^J)^*$,
 \item the set of isomorphic intervals $\{jW^J\mid j\in W_J\}$ partitions $\mathcal{L}_{W}$ and each of them intersects $I(W_J)$ exactly in the element $j$,
 \item the set $\{W_Ji\mid i\in W^J\}$ partitions $\mathcal{L}_{W}$ and each of them induces a graph $G^i$ isomorphic to a subgraph of $G_{I(W_J)}$. \item the edges of $\Cay(W)$ are partitioned into the edges of $\{G^i\mid i \in W^J\}$ and $\{jG(W^J)\mid  j\in W_J\}$. 
 \end{enumerate}
 
 The last item yields that $\Cay(W)$ is a subgraph of $\Cay(W_J) \square G(W^J)$, thus with Lemma~\ref{lem:prodimb} we conclude:

\begin{lemma}\label{lem:intervals}
Let $(W,S)$ be a Coxeter system and $J\subseteq S$. 
We have $$ \iota_k(\Cay(W_J))\iota_{\ell}(G(W^J))\leq \iota_{k+\ell}(\Cay(W)).$$
\end{lemma}
%
%

We call a Coxeter system $(W,S)$ \emph{cube-like} if it admits an abelian parabolic subgroup $W_J$ such that $\iota_0(G(W^J))>0$. A consequence of Lemma~\ref{lem:intervals} together with Theorem~\ref{thm:CFGS88} is:
\begin{proposition}\label{prop:cubelike}
 If $(W,S)$ is cube-like with respect to $J \subseteq S$, then we have $\sigma(\Cay(W))=\lceil\sqrt{\kappa(\Cay(W))}\rceil$ and  $\Cay(W)$ has an induced subgraph of maximum degree $\sigma(\Cay(W))$ on $\frac{|W|}{2}+\iota_0(G(W^J))$ vertices.
\end{proposition}
\begin{proof}
{Denote $d = \kappa(\Cay(W))$, by Proposition \ref{pr:deltasubgraph} we have that $\sigma(\Cay(W)) \geq \lceil\sqrt{d}\rceil$}. 
Since $W_J$ is abelian and minimally generated by $J$, $\Cay(W_J)$ is a cube contained in $\Cay(W)$. Denote its dimension by { $d' = |J|$, we have that $d' \leq  d$}. By Theorem~\ref{thm:CFGS88}, we have that { $\iota_{\lceil\sqrt{d'}\rceil}\Cay(W_J) \geq 2$}. Lemma~\ref{lem:intervals} yields $$\iota_{{\lceil\sqrt{d'}\rceil}+0}(\Cay(W))\geq \iota_{\lceil\sqrt{d'}\rceil}(\Cay(W_J))\iota_{0}(G(W^J))\geq 2\iota_{0}(G(W^J)) > 0.$$ 
As a consequence { $\sigma(\Cay(W)) \leq \lceil\sqrt{d'}\rceil \leq \lceil\sqrt{d}\rceil \leq \sigma(\Cay(W))$; so they are all equalities and we are done.}
\end{proof}
%
%
%
%

A useful feature of cube-like Coxeter groups is that they are closed under products:
\begin{proposition}
 If $(W,S)$ and $(W',S')$ are cube-like, then so is their product $(W\times W',S\times\{e'\}\cup\{e\}\times S')$. Moreover, $\iota_{0}(G(W^J)\square G(W'^{J'}))\geq \iota_0(G(W^J))\iota_{0}(G(W'^{J'}))$.
\end{proposition}
\begin{proof}
 If $J,J'$ yield the two parabolic subgroups witnessing that $(W,S)$ and $(W',S')$ are cube-like, then also $J\times\{e'\}\cup \{e\}\times J'$ generates an abelian parabolic subgroup of $(W\times W',S\times\{e'\}\cup\{e\}\times S')$. The graph $G$ of quotient $W\times W'^{J\times\{e'\}\cup \{e\}\times J'}$ is the Cartesian product $G(W^J)\square G(W'^{J'})$. It follows from Lemma~\ref{lem:prodimb}, that  \[\iota_{0}(G(W^J)\square G(W'^{J'}))\geq \iota_0(G(W^J))\iota_{0}(G(W'^{J'}))>0. \qedhere\]  \end{proof}

We present some necessary and one sufficient criterion for being cube-like:
\begin{proposition}\label{prop:necessary}
Let $(W,S)$ be a Coxeter system with $r$ reflections. If $(W,S)$  is cube-like with respect to $J$, then
\begin{enumerate}
\item $\lceil\sqrt{\kappa(\Cay(W))}\rceil=\lceil\sqrt{|J|}\rceil$,                                                                 \item $r-|J|$ is even,
\item $J$ is inclusion-maximal with respect to generating an abelian subgroup.                                                                           \end{enumerate}
Conversely, if $r-|J|$ is even and the middle layer of $I(W^J)$ is odd, then $(W,S)$ is cube-like with respect to $J$.
\end{proposition}
\begin{proof}

\noindent 1. This is proved implicitly in Proposition~\ref{prop:cubelike}.

\noindent 2. The length of any shortest path from the minimum to the maximum of $\mathcal{L}_W$ is $r$ and for any parabolic subgroup $W_J$, such path can be obtained by first going from the minimum to the maximum of $\mathcal{L}_{W_J}$ and then traversing a translate of the interval $I(W^J)$. Since the diameter of the cube generated by $J$ is $|J|$ we get that $r-|J|$ is the length of $I(W^J)$.

Now, we use the fact 2., that order reversing is an automorphism of $I(W^J)$. If the length of $I(W^J)$ is odd, then this automorphism identifies layers of different parity, hence both parts of a bipartition of $G^J$ will be of the same size and $\iota_0(G^J)=0$.
 
\noindent 3. Whenever $W_J$ is abelian for some $J \subseteq S$, then $J$ corresponds to an independent set in the Coxeter-Dynkin diagram of $W$. When $(W,S)$ is cube-like with respect to $J \subseteq S$, then $J$ corresponds to a maximal independent set. Indeed, if this is not the case, there exists $J \subsetneq J' \subseteq S$ such that $W_{J'}$ is abelian. As a consequence, $G(W_{J'}) = G(W_J) \square G(W_{J' \setminus J})$ and $G(W_{J' \setminus J}) = Q_{|J' \setminus J| },$ a hypercube of dimension $|J' \setminus J| \geq 1$. Hence,  $G(W^{J}) \simeq G(W^{J'}) \square Q_{|J' \setminus J|}$; but this implies that  $\iota_0(G(W^J)) = 0$, a contradiction. 
 
 \medskip
 
 For the sufficient condition, if we have an odd number of layers such that by fact 2. opposite ones are of the same size the bipartition class not containing the middle layer is even, but the one containing the middle layer will be odd. Hence $\iota_0(G^J)>0$.
\end{proof}

From Proposition~\ref{prop:necessary} together with Table~\ref{tab:K} we can infer that the the following Coxeter groups are not cube-like with respect to any $J \subseteq S$:
{$\mathbf{B_{2(n^2+1)}}, \mathbf{B_{2(n+1)^2+1}}$ for $n$ even, $\mathbf{D_{8n^2}}, \mathbf{D_{8n^2+1}}$ for $n \geq 1$, $\mathbf{I_2}(n)$ for $n=0 \mod 2$ and $\mathbf{E_6}$. Indeed, in all these groups there are no $J \subseteq S$ satisfying the necessary conditions of Proposition \ref{prop:necessary}. We will show next that Coxeter groups of type $\mathbf{A_{n}}$ and $\mathbf{I_2}(2k+1)$ are cube-like.
As a consequence any Cayley graph $G$ of them or their products satisfies $\sigma(G)=\lceil\sqrt{\kappa(G)}\rceil$.

The Coxeter system $\mathbf{I_2}(n)$ is $(D_n,\{b,c\})$, where both $b$ and $c$ are generators of order $2$.

\begin{theorem}\label{thm:I_2(2k+1)}
 For any $k\geq 0$ the Coxeter group $\mathbf{I_2}(2k+1)$ is cube-like.
\end{theorem}
\begin{proof}
 The graph $\Cay(\mathbf{I_2}(2k+1))$ is a cycle of length $4k+2$. The maximal abelian parabolic subgroup is generated by a single element $j$, i.e., the maximal cube is an edge. The corresponding quotient $I(W^j)$ is an interval consisting of a single chain of length $2k$. In particular the middle layer is odd and $\iota_0(G^j)>1$, by Proposition~\ref{prop:necessary}.
\end{proof}

The symmetric group $S_{n+1}$ with generators $S = \{(1 2), \, (2 3),\ldots, (n (n+1))\}$ constitutes the Coxeter system of type $\mathbf{A_n}$.  As an example consider $\mathbf{A_3}$. Its illustration in the left of Figure~\ref{fig:B3} shows that this Coxeter group is cube-like, even though it does not satisfy the sufficient condition in Proposition~\ref{prop:necessary}. This exemplifies the construction shown below. 





\begin{theorem}\label{thm:Ancubelike}
 For all $n\geq 0$ the Coxeter system $\mathbf{A_{n}}$ is cube-like with respect to a set $J$ such that  $\iota_0(G(\mathbf{A_{n}}^J))\geq\lceil\frac{n}{2}\rceil!$.
\end{theorem}
\begin{proof}
 We set $J=\{(12),(34), \ldots, (nn+1)\}$ is $n$ is odd and $J=\{(12),(34), \ldots, (n-1n)\}$ otherwise. Clearly the parabolic subgroup generated by $J$ is abelian.  For the proof we identify the permutations with strings of length $n+1$ in the standard way, e.g., $e=[1,2,3,\ldots ,n+1]$. For a permutation $\pi$, its length $\ell(\pi)$ equals the number of pairs that are ordered differently from $[1,2,\ldots, n+1]$. 
 We refer to the two bipartition classes of $\Cay(\mathbf{A_n})$ as \emph{even} and \emph{odd} and correspondingly denote the \emph{parity} of a permutation $\pi$ by $p(\pi) \in \{0,1\}$.
 Let $P^J$ be the poset on $\{1,\ldots, n+1\}$ whose relations are of the form $(i\prec j)$ if $(ij)\in J$.  Now $W^J$ can be seen as the set of \emph{linear extensions} of $P$, i.e., all linear orders on $\{1,\ldots, n+1\}$ that respect the relations prescribed by $P$. 
 
 Let us first consider the case $n$ even. In this setting $P$ consists of the single element $\{n+1\}$ and a disjoint union of chains $1\prec 2, \ldots, n-1\prec n$ called $M$. We label these chains $C_1, \ldots, C_{\frac{n}{2}}$. For the sake of the proof we say that an \emph{arc-diagram} $D$ is a perfect matching of $K_{n}$. Any linear extension $L_M$ of $M$ corresponds to an arc diagram, where each edge is labeled with a chain among $C_1, \ldots, C_{\frac{n}{2}}$.  More precisely, a linear extension of $M$ can be seen as a permutation $\pi$ of $\{1,\ldots,n\}$ such that $i < j$ whenever $\pi(i) \prec \pi(j)$; then the linear extension corresponds to the arc-diagram $D$ with edges  $e_j = (\pi(2j-1),\pi(2j))$ for $1 \leq j \leq n/2$, and the edge $e_j$ is labeled by $C_j$.
  Thus, given one arc diagram $D$ there are $\frac{n}{2}!$ linear extensions with this diagram. Moreover, all of them have the same parity $p(D)$. To see the latter it is sufficient to distinguish how two arcs intersect whose assigned chains are exchanged. We skip this case distinction.
  
Now, there are $n+1$ possible ways to insert $\{n+1\}$ into a given linear extension of $M$ with diagram $D$. Note that $\lceil\frac{n+1}{2}\rceil$ of these have parity $p(D)$ and $\lfloor\frac{n+1}{2}\rfloor$ of these have parity $(p(D)+1) \mod 2$.

Since the number of arc-diagrams, i.e., the number of perfect matchings of $K_{n}$ is odd, for some $p\in\{0,1\}$ there is  one more arc-diagram of parity $p$ than there are of parity $(p+1) \mod 2$. So, take a diagram $D$ of parity $p$. It corresponds to $\lceil\frac{n+1}{2}\rceil\frac{n}{2}!$ linear extensions of parity $p$ and $\lfloor\frac{n+1}{2}\rfloor\frac{n}{2}!$ linear extensions of parity $p+1 \mod 2$. We thus have $\iota_0(G^J)\geq\frac{n}{2}!$.

In the case that $n$ is odd, the same proof works except that $P$ is entirely partitioned into chains of length $2$. The analogous analysis yields $\iota_0(G^J)\geq\frac{n+1}{2}!$. \end{proof}

The results of the present section can be applied to the sensitivity of some Coxeter groups:
\begin{corollary}\label{cor:An}
 Let $G$ be the $n$-vertex Cayley graph of the product $$\mathbf{I_2}(2k_1+1)\times\ldots\times \mathbf{I_2}(2k_i+1)\times \mathbf{A(n_1)}\times\ldots\times \mathbf{A(n_j)}.$$ Then $\sigma(G)=\lceil\sqrt{\kappa(G)}\rceil$ and there exists a set of $\frac{n}{2}+\Pi_{\ell=1}^j(\lceil\frac{n_{\ell}}{2}\rceil!)$ vertices inducing this degree.
\end{corollary}

\bigskip

We proceed to study $\sigma$ for Coxeter groups, where we cannot apply the above strategy.

\begin{theorem}\label{th:small}
 Let $G$ be the Cayley graph of a Coxeter group of type $\mathbf{I_2}(n)$ or $\mathbf{I_2}(n)\times\mathbf{I_2}(n')$. Then $\sigma(G)=\lceil\sqrt{\kappa(G)}\rceil$.
\end{theorem}
\begin{proof}
The Cayley graph of ${\mathbf I_2}(2)$ is a square and, then, $\kappa( {\mathbf I_2}(2))=2$ and $\sigma( {\mathbf I_2}(2)) = 2 = \lceil\sqrt{2}\rceil$.

Let us see that also the product of two even cycles $C_i\square C_j$ has a subgraph on more than half the vertices with max degree at most $2$. If $i = j = 4$, one can take the subgraph consisting of an induced $8$-cycle and a vertex without neighbors in this cycle. So assume that $i > 4$. Take a proper $3$-coloring of $C_j$ with $a,b,x$, such that $x$ is used at least once and such that the neighbors of every vertex colored $x$ are colored differently. Now, in $C_i\square C_j$ every copy of $C_i$ has color $a,b$ or $x$. In every copy of $C_i$ colored with $x$, we always pick the same subgraph of maximum degree $1$ and with more than $i/2$ vertices (we can do this because $i > 4$). In the other copies of $C_i$ we choose one of the two bipartition classes of $C_i$, depending on whether its color is $a$ of $b$. The resulting subgraph has more than half of the vertices and maximum degree $2$. 

Thus, Coxeter groups of the form $\mathbf{I_2}(n)\times \mathbf{I_2}(n')$ are fine, too.
\end{proof}

\begin{center}
\begin{table}[h]
\begin{tabular}{c|r|r|c} 
 
 group&order&$\kappa$&subgraph of degree $\leq 2$ \\ 
 \hline
 $\mathbf{F_4}$&$1152$&$2$&$768$ \\ 
 $\mathbf{H_3}$&$120$&$2$&$85$ \\ 
 $\mathbf{H_4}$&$14400$&$2$&$8624\leq\cdot\leq 9599$\\ 
 $\mathbf{E_6}$&$51840$&$3$&$25926\leq\cdot$\\ 
 $\mathbf{D_4}$&$192$&$3$&$120\leq\cdot\leq 122$\\ 
 $\mathbf{D_5}$&$1920$&$3$&$1004\leq\cdot\leq 1199$\\ 
 $\mathbf{B_3}$&$48$&$2$&$34$ \\ 
 $\mathbf{B_4}$&$384$&$2$&$235\leq\cdot\leq 252$\\ 
 $\mathbf{B_5}$&$3840$&$3$&$1976\leq\cdot\leq 2398$\\ 
 $\mathbf{B_3}\times \mathbf{I_2}(2)$&$192$&$4$&$98\leq\cdot\leq 115$\\
 $\mathbf{B_3}\times \mathbf{I_2}(3)$&$288$&$3$&$150\leq\cdot\leq 175$\\
 $\mathbf{B_3}\times \mathbf{I_2}(4)$&$384$&$3$&$200\leq\cdot\leq 235$\\
  $\mathbf{I_2}(2)\times \mathbf{I_2}(3)\times \mathbf{I_2}(3)$&$144$&$3$&$73\leq\cdot\leq 79$\\
 $\mathbf{I_2}(1)\times \mathbf{I_2}(2)\times \mathbf{I_2}(4)$&$64$&$4$&$33$\\
 $\mathbf{I_2}(1)\times \mathbf{I_2}(3)\times \mathbf{I_2}(4)$&$96$&$3$&$52$\\
\end{tabular}
\caption{Largest subgraphs of maximum degree $\lceil\sqrt{\kappa}\rceil$ in small Coxeter groups.}\label{tab:Q}
\end{table}
\end{center}

%
%

The group $\mathbf{B_n}$ equals the wreath product $\mathbb{Z}_2 \wr_{\{1,\ldots,n\}} S_{n}$ . Equivalently, this group can be seen as the group with elements $2^{[n]} \times S_{n}$, where $2^{[n]}$ denotes the power set of $\{1,\ldots,n\}$, with operation $(A, \pi) \cdot (B, \tau) = (A\, \triangle\, \pi^{-1}(B), \pi \cdot \tau)$ being $\triangle$ the symmetric difference of the two sets,  and generators $S = \{a_1,\ldots,a_{n}\}$ with $a_i = (\emptyset, (i\ i+1))$ for all $i \in \{1,\ldots,n-1\}$ and $a_n = (\{1\},id)$. 

The group $\mathbf{D_n}$ is a subgroup of $\mathbf{B_n}$ of index $2$; it can be seen as the group with elements $E^{[n]} \times S_{n}$, where $E^{[n]}$ denotes the elements in $2^{[n]}$ with an even number of elements, and generators $S = \{a_1,\ldots,a_{n-1},a_n'\}$ with $a_i = (\emptyset, (i \ i+1))$ for all $i \in \{1,\ldots,n-1\}$ and $a_n' = (\{1,2\},(12))$ and . 

\begin{theorem}\label{th:coxeter}

Let $G$ be a Cayley graph of $\mathbf{B_n}$ or $\mathbf{D_n}$. Then $\sigma(G)\leq\lceil\sqrt{\kappa(G)}\rceil+1$.

%
\end{theorem}
\begin{proof}We first observe that $\mathbf{B_n}$ has a bipartite Cayley graph, indeed the bipartition classes are $U_1,U_2 \subseteq 2^{[n]} \times S_{n}$ where
$U_1 = \{(A, \pi)  \, \vert \, |A|$ and $\pi$ have different parity$\}$ and $U_2 = \{(A, \pi) \, \vert \, |A|$ and $\pi$ have the same parity$\}$. The induced subgraph with vertices $\{\emptyset\} \times S_{n}$ is isomorphic to $\Cay(\mathbf{A_{n-1}})$ and, by Corollary \ref{cor:An}, it has an induced subgraph $K$ with more than $n!/2$ vertices and maximum degree $k=\lceil\sqrt{\lceil\frac{n-1}{2}\rceil}\rceil$. 

Now we consider 
$K \cup \{(A, \pi) \in U_1 \, \vert \, A \not= \emptyset \}$, which has $ |K| + \frac{1}{2}(|\mathbf{B_n}| - n!) > \frac{1}{2}|\mathbf{B_n}|$ elements and we are going to prove that
 the maximum induced degree is at most $k+1$.
Take  $(B, \tau) \in K \cup \{(A, \pi) \in U_1 \, \vert \, A \not= \emptyset \}$. If $B \neq \emptyset$, then $(B,\tau) \in U_1$ and $(B, \tau) \cdot a_j  \notin U_1$ for all $j \in \{1,\ldots,n-1\}$; thus,  $(B,\tau)$ has degree at most one. If $B = \emptyset$, then $(B, \tau) \cdot a_j \in K$ for at most $k$ values of $j \in \{1,\ldots,n-1\}$ and, hence, its degree is at most $k+1\leq \lceil\sqrt{\lceil\frac{n}{2}\rceil}\rceil+1=\lceil\sqrt{\kappa(\Cay(\mathbf{B_n}))}\rceil+1$. 


A similar proof works for $\mathbf{D_n}$. 
\end{proof}

 We have shown that several Coxeter groups are tight with respect to the lower bound from Proposition~\ref{pr:deltasubgraph}. Also consider Table~\ref{tab:Q} (see also Figure \ref{fig:B3} for $\mathbf{B_3}$) for further results into this direction that were obtained by computer. All the results from Table~\ref{tab:Q} have been obtained by solving a straight-forward integer linear program in CPLEX, except for $\mathbf{E_6}$ where the linear program exceeded the memory of the computer. In this case an exhaustive search through all pairs of $2$-elements sets of $J$ as candidate for $\mathbf{F}$ for Lemma~\ref{lem:bound} gave the result. Note in particular, while in every cube-like Coxeter group the construction from Lemma~\ref{lem:bound} yields a solution via Theorem~\ref{thm:CFGS88}, $\mathbf{E_6}$ is not cube-like by Proposition~\ref{prop:necessary}. This shows the generality of lattices that are cover subposets of cubes, see Remark~\ref{rem:lattice}. For $\mathbf{E_7}$ and $\mathbf{E_8}$ even this exhaustive method was not feasible by computer.
 However, we believe to have gathered sufficient evidence to dare the following:

\begin{conjecture}\label{conjecture}
 Let $G$ be the Cayley graph of a Coxeter group and $Q_d$ the largest subgraph isomorphic to a cube. Then $G$ contains a set $K$ of more than half the vertices, that induced a subgraph of maximum degree $\lceil\sqrt{d}\rceil$, i.e., $\sigma(G)=\lceil\sqrt{\kappa(G)}\rceil$.
\end{conjecture}

\section{The absence of cubes}

In a sense most of the paper so far has been about Huang's lower bound (Proposition~\ref{pr:deltasubgraph}) being tight, i.e., if a bipartite Cayley graph contains a largest cube $Q_d$, then there is an induced subgraph of maximum degree at most $\lceil\sqrt{d}\rceil$ on more than half the vertices. 

However, we do not want to give the wrong impression that this lower bound is tight in general bipartite Cayley graphs. In this section we provide families of cube-free graphs and with unbounded sensitivity. Observe that if a graph has girth at least $6$, then it does  not contain non-trivial hypercubes, since these have a $4$-cycle.

An $(n,d,\lambda)$-graph is a $d$-regular graph (which might have loops) on $n$ vertices in which all nontrivial (different from $d$) eigenvalues have absolute value at most $\lambda$.

The following Theorem \ref{thm:sensitcover} was provided to us by Noga Alon. For its statement we need the notion of Kronecker double cover introduced at the end of Section~\ref{sec:tight} and for its proof we use Lemma \ref{lm:eml}, which is a direct consequence of the so-called expander mixing lemma (see, e.g.,~\cite{alo-88,alo-16},~\cite[Section 5]{Haemers} or~\cite[Lemma 2.1]{anorak}).
One consequence of Theorem \ref{thm:sensitcover} will be Corollary~\ref{cor:pp}, which slightly improves the  bound we obtained in an earlier version of this paper. 

\begin{lemma}\label{lm:eml} Let $K$ be an $(n,d,\lambda)$-graph and consider $S, T \subseteq V(K)$ with $|S|+|T| = n$. Then, 
\[ e(S,T) \geq (d-\lambda)\frac{|S||T|}{n}, \]
where $e(S, T)$ is the number of (ordered) edges $uv$ with $u \in S, v \in T$.
In particular, there is a vertex of $S$ that has at least $(d-\lambda)\frac{|T|}{n}$ neighbors in $T$.
\end{lemma}
\begin{proof}The expander mixing lemma  asserts that 
\[ \left| e(S,T) - d \frac{|S||T|}{n} \right| \leq \lambda \sqrt{|S||T|\left(1 - \frac{|S|}{n}\right) \left(1 - \frac{|T|}{n} \right)}. \]
 In particular, if $|S|+|T| = n$, this implies that 
\[ e(S,T) \geq (d-\lambda)\frac{|S||T|}{n} \]
and thus there is a vertex of $S$ that has at least $(d-\lambda)\frac{|T|}{n}$ neighbors in $T$.
\end{proof}

\begin{theorem}\label{thm:sensitcover} Let $G$ be the Kronecker double cover of an $(n,d,\lambda)$-graph, then $\sigma(G) > (d-\lambda)/2$.
\end{theorem}
\begin{proof}Take $H$ such that $G = H \times K_2$, i.e., $G$ is the Kronecker double cover of $H$. Take a set $U$ with more than $n$ vertices in $G$. Since $|U| > n$, then $U$ must contain vertices in both
vertex classes of $G$. Take $S,T$ be the nonempty sets of vertices of $H$ corresponding to the vertices of $U$ in each of the two vertex classes of $G$. Without loss of generality we assume $|T| > n/2$. If $|T| = n$, then $T = V(H)$ every vertex in $S$ has its $d$ neighbors in $T$. Otherwise, we consider $S' \subseteq S$ with $|S'| = n - |T|$.  Applying  Lemma \ref{lm:eml} with $S'$ and $T$ we get that there is a vertex of $S'$ that has at least $(d-\lambda)\frac{|T|}{n} > (d-\lambda)/2$ neighbors in $T$. The desired result follows from the definition of double cover.
\end{proof}

Consider the polarity graph of the Desarguesian projective plane $P(2,q)$ (with loops), which is a $(q^2+q+1, q+1, \sqrt{q})$-graph. Note that this graph is not transitive (and, thus, not a Cayley graph) in general. Now, its Kronecker double cover is the \emph{Levi graph}, i.e., point-line incidence graph, of $P(2,q)$, which we denote by $L_q$. It is known that $L_q$ has girth $6$. 
 Moreover, $L_q$ is the Cayley graph of $D_{q^2+q+1}$ with respect to a set of $q+1$ involutions, see~\cite[Theorem 1]{Loz-11}. As a direct consequence of Theorem \ref{thm:sensitcover} we have the following result providing a family of cube-free $(q+1)$-regular Cayley graphs and unbounded sensitivity.

\begin{corollary}\label{cor:pp}
The graph $L_q$ satisfies $\sigma(L_q)> (q+1 - \sqrt{q})/2.$
\end{corollary}

Note that we do not need the projective plane to be Desarguesian for the proof to work (see, e.g., \cite{TaitTimmons}).

Using the list of small vertex-transitive graphs~\cite{list,royle_data,HoG}, we verified that each vertex-transitive, bipartite $G$ on at most 47 vertices and with girth at least $6$ has $\sigma(G)\leq 2$. Also compare sequences A185959 and A006800 in~\cite{EOIS} for the numbers of Cayley and transitive graphs, respectively.
In particular, examination by computer shows that $\sigma(L_q)=2$ for $q\leq 4$ and $\sigma(L_q)=3$ for $q=5,7$. Indeed, $L_2$ is the well-known \emph{Heawood graph} and yields the smallest transitive bipartite graph with girth $6$ and $\sigma=2$. The 62-vertex Levi graph of the Desarguesian projective plane $P(2,5)$ is the smallest transitive graph with $\sigma=3$ that we know of. In particular, Corollary~\ref{cor:pp} yields $\sigma(L_{8})\geq 4$ and the computer finds that this is an equality.

Theorem~\ref{thm:sensitcover} can also be used to provide bipartite Cayley graphs with high sensitivity and arbitrarily high girth. Indeed, a famous construction of Ramanujan graphs by~\cite{lub-88} gives families of unbounded girth. For $p,q$ distinct primes congruent to $1$ modulo $4$ they construct a $(p+1)$-regular Cayley graph $X^{p,q}$ in which all nontrivial eigenvalues have absolute value at most $2\sqrt{p}$. Its properties depend on the Legendre symbol of $p$ and $q$. Namely, if $\left(\frac{p}{q}\right)=-1$, then $X^{p,q}$ is a bipartite Cayley graph of PGL$(2,q)$ (of order $q(q^2-1)$) with girth at least $4\log_pq-\log_q4$.
If $\left(\frac{p}{q}\right)=1$, then $X^{p,q}$ is a non-bipartite Cayley graph of PSL$(2,q)$ (of order $\frac{q(q^2-1)}{2}$) with girth at least $2\log_pq$. In this case, we denote by $Y^{p,q}$ the Kronecker double cover of $X^{p,q}$ 
Since the Kronecker double cover of a Cayley graph is a bipartite Cayley graph (see Remark \ref{rm:cayleydouble}), $Y^{p,q}$ is a bipartite Cayley graph of arbitrary high girth and  Theorem~\ref{thm:sensitcover} yields: 

\begin{corollary}\label{cor:Ramanujan}
The graph $Y^{p,q}$ satisfies $\sigma(Y^{p,q})> (p+1)/2 - \sqrt{p}.$
\end{corollary}

Interestingly, the bound of Corollary \ref{cor:Ramanujan} depends on the degree of the graph, but not on the number of vertices of the graph $Y^{p,q}$.

There are many graphs with similar properties to the ones in the preceding corollaries. Indeed, by a result of~\cite{alo-94} for every $0<\delta<1$ there exists $c_{\delta}$ such that
for any group $\Gamma$ of order $n$, the Cayley graph $\Cay(\Gamma,C)$ with
respect to a random set $C\subseteq\Gamma$ of size $c_{\delta}\log n$
has $\lambda\leq(1-\delta)d$ almost surely. Moreover, it is known that the girth of $\Cay(\Gamma,C)$ is large with high probability for many groups, see~\cite{gam-09}. As a direct consequence of Theorem~\ref{thm:sensitcover} we have the following lower bound on $\sigma$.

\begin{corollary}\label{cor:random}
For every $0 < \delta < 1$ there is a $c_\delta > 0$ such that the
following holds. Let $\Gamma$ be any group of order $n$ and let $G = \Cay(\Gamma \times \mathbb Z_2,C \times \{1\})$ be the Cayley graph of $\Gamma\times \mathbb Z_2$ with
respect to a set $C$ of $c_\delta \log(n)$
random elements of $\Gamma$. Then $\sigma(G) \geq \delta d / 2$ almost surely.
\end{corollary}

One may wonder if the above constructions give \emph{minimal} Cayley graphs, i.e., with respect to inclusion-minimal generating sets. However,  we observe the following:

\begin{lemma}\label{lem:min}
 Let $G$ be an $(n,d,\lambda)$-graph. If $G \times K_2$ is a minimal Cayley graph, then $\lambda\geq d-4$.
 
\end{lemma}
\begin{proof}
 Since $G \times K_2 =\Cay(\Gamma,C)$ is minimal, removing any generator from $C$ creates a disconnected graph. Thus, there is a $1$- or $2$-factor $F$, whose edge-removal disconnects $G$. Let $S$ be the set of vertices in one of these connected components and $T = V(G \times K_2) - S$  (observe that $|T| \geq n$). Denote by $0,1$ the vertices of $K_2$ and write $S_i = S \cap (G \times \{i\})$ and $T_i = T \cap (G \times \{i\})$  for $i = 0,1$. We have that $|S_0| = |S_1|$, $|T_0| = |T_1|$ and, then, $|S_0| + |T_1| = n$. Consider $S' = \{s \, \vert \, (s,0) \in S_0\} \subseteq V(G)$ and $T' = \{t \, \vert \, (t,1) \in T_1\} \subseteq V(G)$  (observe that $|T'| \geq n/2$).   Thus, we get $e(S',T')\leq 2 |S'|$. On the other hand, by Lemma \ref{lm:eml}, we get that 
 \[ e(S',T') \geq (d-\lambda)\frac{|S'||T'|}{n} \geq (d-\lambda)\frac{|S'|}{2}.\]
 The result follows from both inequalities.
\end{proof}
%
%

The Levi graphs of projective planes from Corollary~\ref{cor:pp} are Kronecker double covers of $(q^2+2+1,q+1,\sqrt{q})$ graphs. Hence, by Lemma~\ref{lem:min}, we deduce that $L_q$ is not a minimal Cayley graph for $q > 5$.  So Levi graphs are not minimal Cayley graphs whenever $\sigma > 3$. 
Similarly, one can see that the graphs $Y^{p,q}$ in Corollary \ref{cor:Ramanujan} with $p > 5$ cannot provide minimal Cayley graphs. 
Since the sensitivity of a $d$-regular bipartite graph is upper bounded by $d-1$, we have that this construction does not yield minimal Cayley graphs with $\sigma >5$. 

%

So while Corollaries~\ref{cor:pp} and~\ref{cor:Ramanujan} give rise to families of Cayley graphs of $\kappa \equiv 1$ and unbounded $\sigma$, one might still believe that minimal Cayley graphs could satisfy Huang's lower bound with tightness. 
However, the M\"obius-Kantor graph $G(8,3)$ is bipartite and the Cayley graph $\Cay(P_1,\{X,Y,Z\})$, where
$$X=\begin{pmatrix}0&1\\1&0\end{pmatrix}Y=\begin{pmatrix}0&-i\\i&0\end{pmatrix}Z=\begin{pmatrix}1&0\\0&-1\end{pmatrix}$$ and $P_1=\{\pm I,\pm iI,\pm X,\pm iX, \pm Y, \pm iY, \pm Z, \pm iZ\}<SU(2)$ is the \emph{(first) Pauli group}, see the left part of Figure~\ref{fig:Mobius}. This group can also be described as central product of $\mathbb{Z}_4$ with $D_4$. 
While $G(8,3)$ has girth $6$ one can check that $\sigma(G(8,3))=2>1=\lceil\sqrt{\kappa(G(8,3))}\rceil$. 

\begin{figure}[ht]
\includegraphics[width=.9\textwidth]{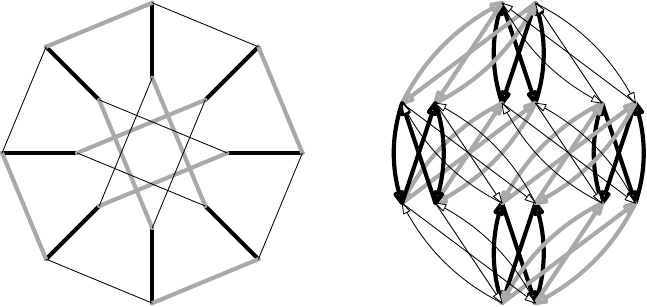}
\caption{Two Cayley graphs of the Pauli group. Left: the M\"obius-Kantor graph $G(8,3)$ as $\Cay(P_1,\{X,Y,Z\})$. Right: the lexicographic product $Q_3[\overline{K_2}]$ as $\Cay(P_1,\{iI,-iX,-iZ\})$, where $iI$ corresponds to the thick gray arcs.} \label{fig:Mobius}
\end{figure}

Indeed $G(8,3)$ is also isomorphic to both $\Cay(M_{16},C)$ and $\Cay(QD_{16},C)$, where $M_{16}=\{x^ry^s\mid x^8=y^2=e, yx=x^5y\}$ is the modular group of order $16$, $QD_{16}=\{x^ry^s\mid x^8=y^2=e, yx=x^3y\}$ is the quasidihedral group of order $16$, and  $C = \{x,y\}$. Yet another way of representing the M\"obius-Kantor graph is as the dihedrant $\Cay(D_8,\{b,ab,a^3b\})$. However, this generating set is not minimal. For further information on this remarkable graph, see~\cite{Pisanski}. 

Another example is the lexicographic product $Q_3[\overline{K_2}]$. One way of representing  this graph is as $\Cay(P_1,\{iI,-iX,-iZ\})$ where again $P_1$ is the Pauli group, see the right part of Figure~\ref{fig:Mobius}. Another representation is $Q_3[\overline{K_2}]\cong \Cay(\mathbb{Z}_2^2\times\mathbb{Z}_4,\{(1,0,1),(0,1,1),(0,0,1)\})$. One calculates $\sigma(Q_3[\overline{K_2}])=4>2=\lceil\sqrt{\kappa(Q_3[\overline{K_2}])}\rceil$. We believe that the family of graphs $G_m := Q_m[\overline{K_2}]$ with $m \in \mathbb Z^+$ is a good candidate to provide minimal Cayley graphs where the difference $\sigma(G_m)-\sqrt{\kappa(G_m)}$ is unbounded. 

All this leads to a question analogous to $\chi$-boundedness~\cite{Gya-87} or $\tau$-boundedness~\cite{Fel-20}.

\begin{question}[$\sigma$-boundedness]\label{quest:min}
 Is there a function $f$ such that for every minimal bipartite Cayley graph  $G$, we have $\sigma(G)\leq f(\kappa(G))$?
\end{question}


\section{Conclusions}
Most of the paper is about Huang's lower bound (Proposition~\ref{pr:deltasubgraph}) being tight, i.e., if a bipartite Cayley graph contains a largest cube $Q_d$, then there is an induced subgraph of maximum degree at most $\lceil\sqrt{d}\rceil$ on more than half the vertices. 
We show that this holds for some dihedrants, the star graphs, and some tight groups, where these results can be seen as proving \emph{insensitivity}. We further prove the lower bound to be tight for large classes of Coxeter groups, and conjecture it for general Coxeter groups (Conjecture~\ref{conjecture}).
On the other hand we show that there are also cube-free graphs of unbounded sensitivity, e.g., Levi graphs of projective planes. 
 A curiosity is that the latter class as well as the first family of insensitive graphs are dihedrants with respect to non-minimal generating sets. While we have provided insensitive Cayley graphs with respect to minimal generating sets, 
 it remains open if there are bipartite Cayley graphs with respect to a minimal generating set that have bounded $\kappa$ and unbounded $\sigma$ (Question~\ref{quest:min}).

Further, we believe that the $k$-imbalance $\iota_k$, i.e., the coloring parameter associated to sensitivity $\sigma$ deserves further investigation. Indeed, apart from cube-like Coxeter groups and in particular $Q_d$, our results on star graphs and  tight groups can be read in terms of this stronger parameter.

\bigskip

Let us finally conclude with some thoughts on non-bipartite Cayley graphs. Since many things already do not work in the bipartite case, let us go back to abelian groups. 
The result of~\cite{potechin2020conjecture} gives a lower bound on the induced maximum degree when more than half of the vertices are taken, but in a non-bipartite Cayley graph $\alpha$ is less than half the vertices.
We show that the stronger version is false, i.e., there abelian groups with Cayley graphs of unbounded degree but $\sigma\equiv 1$. 

\begin{theorem}\label{thm:Z3}
 We have $\sigma(\Cay(\mathbb{Z}_3^r,\{(1,\ldots,0),(0,1,\ldots,0),\ldots,(0,\ldots,1)\}))=1$, for all positive integers $r$.
\end{theorem}
\begin{proof}
 First note that $\alpha(\Cay(\mathbb{Z}_3^r,\{(1,\ldots,0),(0,1,\ldots,0),\ldots,(0,\ldots,1)\}))=3^{r-1}$. We show that there is a set of $3^{r-1}+1$ vertices inducing degree $1$ whose complement contains a maximum independent set, by induction on $r$. While the case $r=1$ is trivial, let us consider $r>1$ and take a set $A$ of $3^{r-2}+1$ vertices inducing degree $1$ and disjoint independent set $B$ of size $3^{r-2}$ both in $\Cay(\mathbb{Z}_3^{r-1},\{(1,\ldots,0),(0,1,\ldots,0),\ldots,(0,\ldots,1)\})$. Our solution for $\Cay(\mathbb{Z}_3^{r},\{(1,\ldots,0),(0,1,\ldots,0),\ldots,(0,\ldots,1)\})$ is $A'=A\times\{0\}\cup B\times\{1,2\}$. This set has size $3^{r-2}+1+2(3^{r-2})=3^{r-1}+1$ and induces degree $1$. Moreover, the set $B\times\{0\}\cup (B+(1,\ldots,0))\times\{1\}\cup (B+(2,\ldots,0))\times\{2\}$ induces a maximum independent set complementary to $A'$. \end{proof}

We do not know if there is a family of tripartite Cayley graphs playing the role of hypercubes with respect to $\sigma$. More generally, we wonder:
\begin{question}
 Is there an infinite family $\mathcal{G}$ of non-bipartite (minimal) sensitive Cayley graphs, i.e., is there a function $f$ such that $d\leq f(\sigma(G))$ for all $d$-regular $G\in \mathcal{G}$?
\end{question}



\subsubsection*{Acknowledgments} We thank Noga Alon for several discussions concerning Section~7. In particular, he provided us with Theorem \ref{thm:sensitcover} that led to Corollaries \ref{cor:Ramanujan},  \ref{cor:random} and an improvement of an earlier version of Corollary \ref{cor:pp} as well as several comments on minimal Cayley graphs which led us to Lemma \ref{lem:min}. We thank Gordon Royle for sharing with us the list of transitive graphs on up to 47 vertices \cite{list,royle_data,HoG}, Gabriel Verret to pointing to us to the list of cubic vertex-transitive graphs generated via~\cite{PPV13,PPV15}, and Anurag Bishnoi helping us with a first proof of Corollary~\ref{cor:pp}. We also would like to thank the anonymous referees for their valuable comments and suggestions.

The first author was partially supported by the Universidad de La Laguna MASCA project. The second author was partially supported by the French \emph{Agence nationale de la recherche} through project ANR-17-CE40-0015 and by the Spanish \emph{Ministerio de Econom\'ia,
Industria y Competitividad} through grant RYC-2017-22701. Both authors were partially supported by the Spanish MICINN through grant PID2019-104844GB-I00 and by the Universidad de La Laguna MACACO project.

\bibliography{lit}
\bibliographystyle{my-siam}
\end{document}